\documentclass[10pt,letterpaper]{amsart}

\usepackage{graphicx,hyperref}
\usepackage[hmargin=1in, vmargin=1in]{geometry}
\renewcommand{\theequation}{\thesection.\arabic{equation}}

\theoremstyle{plain}
\newtheorem*{main}{Main Theorem}
\newtheorem{thm}{Theorem}[section]
\newtheorem{lem}{Lemma}[section]
\newtheorem{prop}[lem]{Proposition}

\theoremstyle{definition}
\newtheorem{defn}{Definition}
\newtheorem{rmk}{Remark}

\renewcommand{\[}{\begin{equation}\notag\begin{aligned}}
\renewcommand{\]}{\end{aligned}\end{equation}}
\newcommand{\beq}[1]{\begin{equation}\label{#1}\begin{aligned}}
\newcommand{\beqeps}[1]%
  {\addtocounter{equation}{1}\begin{equation}\label{#1}\tag{$\theequation\epsilon$}\begin{aligned}}

\newcommand{\p}{\begin{pmatrix}}
\newcommand{\pp}{\end{pmatrix}}

\newcommand{\din}{\mathrm{in}}
\newcommand{\dout}{\mathrm{out}}

\newcommand{\myflow}[1]{\underset{\eqref{#1}}{\bullet}}

\author{Ting-Hao Hsu}
\address{Department of Mathematics\\ The Ohio State University\\ Columbus, OH 43210}
\email{hsu.296@osu.edu}
\title{Viscous  singular shock profiles for the Keyfitz-Kranzer system}


\begin{document}
\maketitle
\begin{center}
\today
\end{center}

\begin{abstract}
It was shown by Schecter \cite{Schecter:2004},
using the methods of Geometric Singular Perturbation Theory,
that the Dafermos regularization $u_t+f(u)_x= \epsilon tu_{xx}$
for the Keyfitz-Kranzer system
admits an unbounded family of solutions.
Inspired by that work,
in this paper we provide a more intuitive approach
which leads to a stronger result.
In addition to the existence of viscous profiles,
we also prove the weak convergence
and show that the maximum of the solution is of order $\epsilon^{-2}$.
This asymptotic behavior is distinct from
that obtained in the author's recent work \cite{Hsu:2015a})
on a system modeling two-phase fluid flow,
for which the maximum of the viscous solution is of order $\exp(\epsilon^{-1})$.
\end{abstract}

\section{Introduction}
\label{sec_intro}
The Keyfitz-Kranzer system \beq{claw_u12}
  &u_{1,t}+ (u_1^2-u_2)_x= 0\\
  &u_{2,t}+ (\tfrac13u_1^3-u_1)_x= 0
\] was first introduced in \cite{Keyfitz:1989,Kranzer:1990}.
It is a strictly hyperbolic, genuinely nonlinear system of conservation laws.
A significant feature is that this model
provides an example for \emph{singular shocks}.
A singular shock,
roughly speaking,
is a measure which contains delta functions
and is the weak limit of some approximate solutions.
For details of the definition, we refer to \cite{Sever:2007,Keyfitz:2011}.

The existence of singular shocks for \eqref{claw_u12}
was proved by Keyfitz and Kranzer \cite{Keyfitz:1995}.
In that work, 
for certain Riemann data \beq{ic_riemann_u12}
  (u_1,u_2)(x,0)
  = \begin{cases}
    (u_{1L},u_{2L}),&x<0,\\
    (u_{1R},u_{2R}),&x>0,
  \end{cases}
\]
they construct approximate solutions of
the regularized system via Dafermos regularization \beqeps{dafermos_u12}
  &u_{1,t}+ (u_1^2-u_2)_x= \epsilon t u_{1,xx}\\
  &u_{2,t}+ (\tfrac13u_1^3-u_1)_x=  \epsilon t u_{2,xx}.
\]
In particular, they proved that there are approximate solution of \eqref{dafermos_u12}
that converges to a step function away from the discontinuity as $\epsilon\to 0$,
and approaches a combination of delta functions near the discontinuity.

A family of exact solutions of \eqref{dafermos_u12}, rather than approximate solutions,
is called a \emph{viscous profile} of \eqref{claw_u12}.
The existence of viscous profiles of \eqref{claw_u12}
was proved in \cite{Schecter:2004}
using \emph{Geometric Singular Perturbation Theory} (GSPT).
In that work,
existence of solutions of \eqref{dafermos_u12} and \eqref{ic_riemann_u12} were proved,
and the solutions approach infinity near the discontinuity as $\epsilon\to 0$,
but convergence of solutions was not considered.
We enhance that pioneering work in the following respects:
First, we simplify the process of blowing-up in \cite{Schecter:2004},
and construct solutions in a more intuitive way.
Second, we prove the weak convergence of the solutions,
which confirms the conjecture in \cite{Kranzer:1990}.

The system \eqref{claw_u12} can be derived
from a single space dimensional model for isentropic gas dynamics equations \beq{claw_gas}
  &\rho_t+ (\rho u)_x= 0\\
  &(\rho u)_t+ (\rho u^2+ \rho^\gamma)_x= 0
\] with $\gamma=1$,
which corresponds to isothermal gas dynamics.
By subtracting $u$ times the first equation in \eqref{claw_gas}
from the second equation,
one obtains \eqref{claw_u12}
with $u_1= u$
and $u_2= \tfrac12u^2- \log\rho$ (see \cite{Keyfitz:2011})      .
This means that \eqref{claw_u12} is equivalent to
the isothermal gas dynamics \eqref{claw_gas} for smooth solutions,
but conservation of mass and momentum has been replaced by
conservation of velocity and a quantity that is an entropy for the original system.

The system \eqref{claw_gas} with
any $\gamma$ between $1$ and $5/3$ was considered in \cite{Keyfitz:2012},
and the existence of viscous profiles for singular shock was also proved.
Some other generalizations of \eqref{claw_u12}
were systematically analyzed in \cite{Sever:2007}.

In Section \ref{sec_main}, we state our main result.
In Sections \ref{sec_cpt}
we sketch the construction of the solutions.
In Section \ref{sec_GSPT},
we recall some tools in geometric singular perturbation theory,
including Fenichel's Theorems and the Exchange Lemma.
In Sections \ref{sec_singular_config}
we verify that the conditions of GSPT for our construction.
The proof for the Main Theorem is given in Section \ref{sec_pf_main}.

\section{Main Result}
\label{sec_main}
In standard notation for conservation laws,
we write \eqref{claw_u12} as
\beq{claw_u}
  u_t+f(u)_x=0,
\] where $u=(\beta,v)$,
and write Riemann data for Riemann problems
in the form \beq{bc_riemann}
  u(x,0)= u_L+ (u_R-u_L) \mathrm{H}(x),
\]
where $\mathrm{H}(x)$ is the step function taking value $0$ if $x<0$; $1$ if $x>0$.

We study the systems that approximate \eqref{claw_u}
via the Dafermos regularization: \beqeps{dafermos_u}
  u_t+f(u)_x=\epsilon t u_{xx}
\] for small $\epsilon> 0$.
Using the self-similar variable $\xi = x/t$,
the system is converted to \beqeps{dafermos_similar}
  -\xi \frac{d}{d\xi} u+ \frac{d}{d\xi}\big(f(u)\big)
  = \epsilon \frac{d^2}{d\xi^2}u,
\] and the initial condition \eqref{bc_riemann} becomes \beq{bc_u_infty}
  u(-\infty)=u_L,\; u(+\infty)=u_R.
\]
The system \eqref{dafermos_similar} is equivalent to \beqeps{dafermos_similar_xi}
  &-\epsilon u_\xi= f(u)- \xi u- w\\
  &w_\xi= - u\\
\] or, up to a rescaling of time, \beqeps{sf_u}
  &\dot u= f(u)- \xi u- w\\
  &\dot{w}= -\epsilon u\\
  &\dot\xi= \epsilon.
\] The time variable in \eqref{sf_u}
is implicitly defined by the equation of $\dot\xi$.
When $\epsilon=0$, \eqref{sf_u} is reduced to \beq{fast_u}
  &\dot u= f(u)-\xi u- w\\
  &\dot{w}= 0,\;
  \dot\xi= 0.
\]

Returning to the $(u_1,u_2)$ notation, 
the system \eqref{sf_u} is written as \beqeps{sf_u12}
  &\dot{u}_1= u_1^2-u_2-\xi u_1-w_1\\
  &\dot{u}_2= \tfrac13u_1^3-u_1-\xi u_2- w_2\\
  &\dot{w}_1=- \epsilon u_1,\;
  \dot{w}_2= -\epsilon u_2,\;
  \dot\xi= \epsilon.
\] and \eqref{fast_u} becomes \beq{fast_u12}
  &\dot{u}_1= u_1^2-u_2-\xi u_1-w_1\\
  &\dot{u}_2= \tfrac13u_1^3-u_1-\xi u_2- w_2\\
  &\dot{w}_1=0,\;
  \dot{w}_2= 0,\;
  \dot\xi= 0.
\] At any equilibrium $u_0=(u_{10},u_{20})$ of \eqref{fast_u12},
the eigenvalues for the linearized system are \beq{def_lambda_pm}
  \lambda_-(u_0)= u_{10}-1,\quad
  \lambda_+(u_0)= u_{10}+1.
\]

\begin{figure}[t]
\centering
\begin{parbox}{.48\textwidth}{\centering
\includegraphics[trim = 2cm 6.5cm 2cm 8.2cm, clip, width=.46\textwidth]{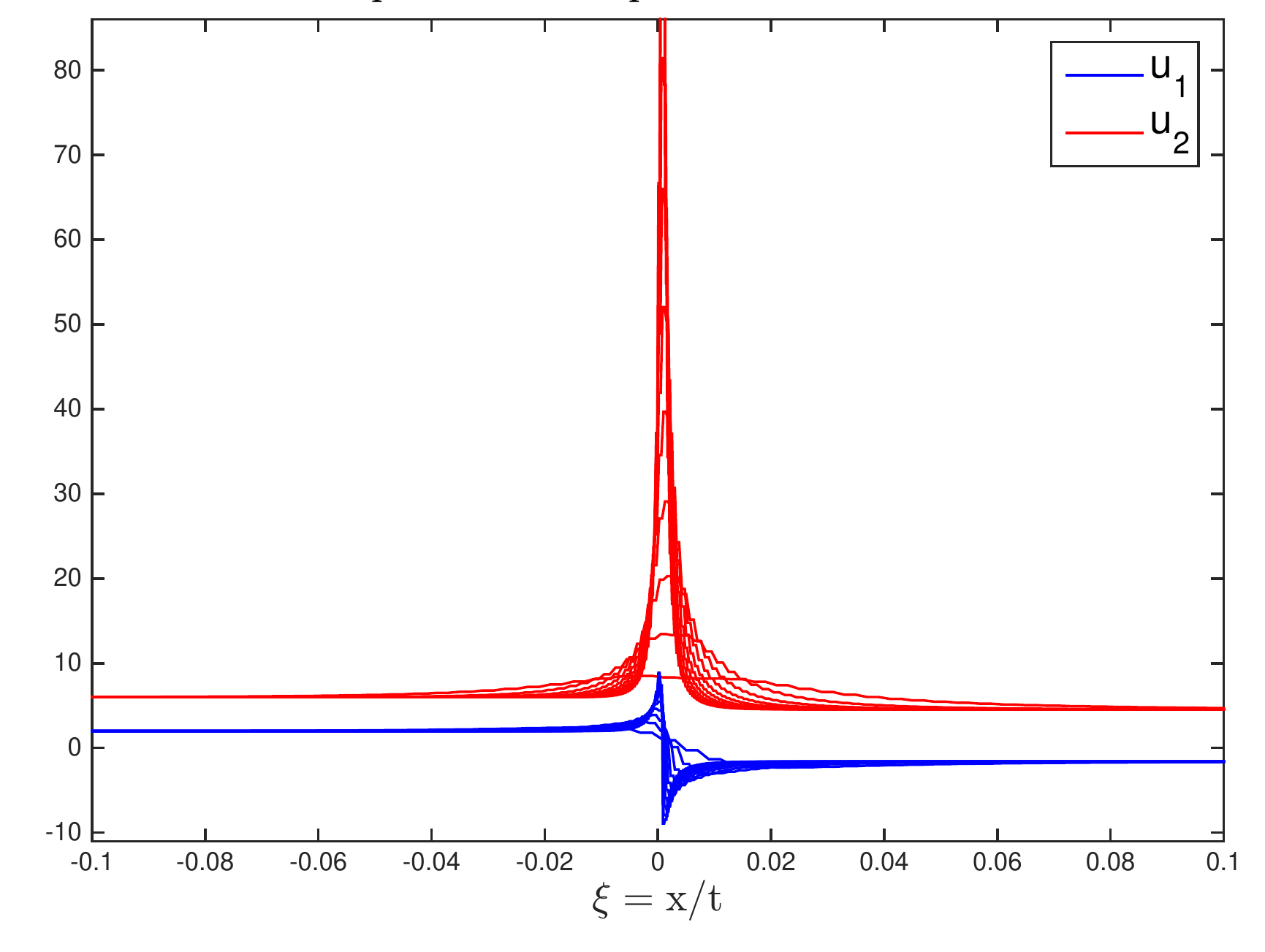}
}\end{parbox}
\begin{parbox}{.48\textwidth}{\centering
\includegraphics[trim = 2cm 6.5cm 2cm 8.2cm, clip, width=.46\textwidth]{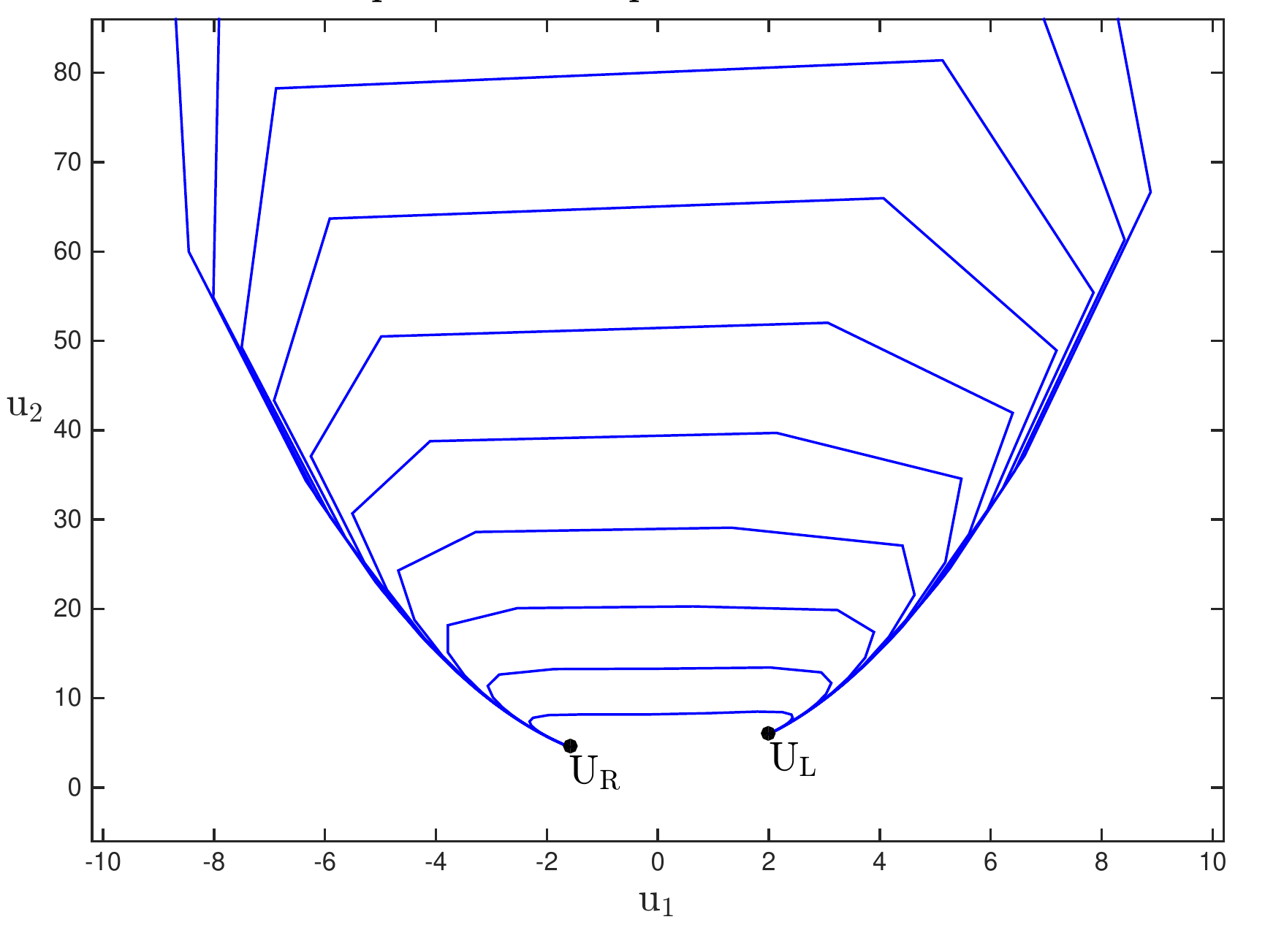}
}\end{parbox}
\caption{
Numerical solutions for the Riemann data
$u_L=(1,6)$ and $u_R=(-1.6.4.56)$ using a Lax-Friedrichs scheme
up to $50,000$ steps with $\mathrm{CFL}=0.05$.}
\label{fig_LF}
\end{figure}

\begin{main}
Consider the Riemann problem \eqref{claw_u12} and \eqref{ic_riemann_u12}.
Let \begin{subequations}\label{def_speed}\begin{align}
  &s=\frac{f_1(u_L)-f_1(u_R)}{u_{1L}-u_{1R}}\label{def_s}\\
  &w_L= f(u_L)- su_L,\; w_R= f(u_R)- su_R\label{def_wLwR}\\
  &e_0= w_{2L}- w_{2R}.\label{def_e0}
\end{align}\end{subequations}
Assume \begin{enumerate}
  \item[$\mathrm{(H1)}$]
  $\mathrm{Re}(\lambda_\pm(u_R))< s< \mathrm{Re}(\lambda_\pm(u_L))$.
  \item[$\mathrm{(H2)}$]
  $e_0>0$.
\end{enumerate}
Then there exists a Dafermos profile for a singular shock from $u_L$ to $u_R$.
That is, for each small $\epsilon>0$,
there is a solution $\tilde{u}_\epsilon(\xi)$ of \eqref{dafermos_similar} and \eqref{bc_u_infty},
and this solution becomes unbounded as $\epsilon\to 0$.
Indeed, \begin{subequations}\begin{align}
  &\max_{\xi}\pm\tilde{u}_{1\epsilon}(\xi)= \Big(
    \omega_0+ o(1)
  \Big)\epsilon^{-1}
  \label{est_u1max}\\
  &\max_{\xi}\tilde{u}_{2\epsilon}(\xi)
  = \left(\kappa_0^2+o(1)\right)\epsilon^{-2},\label{est_u2max}
\end{align}\label{est_umax}
\end{subequations}
as $\epsilon\to 0$,
where $\kappa_0$ and $\omega_0$ are positive constant
defined later in \eqref{def_kappa0} and \eqref{def_omega0}.
Moreover, if we set $u_\epsilon(x,t)=\tilde{u}_\epsilon(x/t)$,
then $u_\epsilon(x,t)$ is a solution of \eqref{dafermos_u12}
and \begin{subequations}\begin{align}
  &u_{1\epsilon}
  \rightharpoonup u_{1L}+ (u_{1R}-u_{1L})\mathrm{H}(x-st)\label{limit_u1}\\
  &u_{2\epsilon}
  \rightharpoonup u_{2L}+ (u_{2R}-u_{2L})\mathrm{H}(x-st)
  + \tfrac{e_0}{\sqrt{1+s^2}}t\delta_{\{x=st\}}\label{limit_u2}
\end{align}\label{limit_u12}\end{subequations}
in the sense of distributions.
\end{main}

The notation $t\delta_{\{x=st\}}$ in \eqref{limit_u2}
denotes the linear functional defined by \beq{def_delta}
  \langle t\delta_{\{x=st\}},\varphi\rangle
  = \int_0^\infty t\varphi(st,t)\sqrt{1+s^2}\;dt.
\] The weight $\sqrt{1+s^2}$ is the arc length of the parametrized line $\{x=st\}$,
so that the definition of the functional
is independent of parametrizations.

A set of sample data for which $\mathrm{(H1)}$ and $\mathrm{(H2)}$ hold is \beq{data_sample}
  u_L= (2,6),\; u_R= (-1.6,4.56),
\] for which \eqref{def_speed} gives $s=0$ and $e_0=0.423$.
A numerical solution
for this Riemann data
using a finite difference scheme
is shown in Fig \ref{fig_LF}.
Observe that both $u_1$ and $u_2$ appear to grow unboundedly near the shock.
This is consistent with the theorem.

\section{Compactification and Desingularization}
\label{sec_cpt}

To find solutions of \eqref{sf_u12} connecting $u_L$ and $u_R$,
we first consider the limiting system \eqref{fast_u12}
with $(w_1,w_2,\xi)= (w_{1L},w_{2L},s)$ and $(w_{1R},w_{2R},s)$,
where $s$, $w_L$ and $w_R$ are as defined in \eqref{def_speed}.

\begin{prop}\label{prop_sss}
Assume $\mathrm{(H1)}$.
Then there exists
a unique solution of \eqref{fast_u12}
of the form $\gamma_1(\sigma)=(u^{(1)}(\sigma),w_L,s)$ satisfying \[
  \lim_{\sigma\to-\infty}u^{(1)}(\sigma)= u_L,\quad
  \lim_{\sigma\to0^-}\left(
    u_2^{(1)},
    \frac{u_1^{(1)}}{\sqrt{u_2^{(1)}}}
  \right)(\sigma)
  = \left(+\infty,\sqrt{3-\sqrt{3}}\right)
\] and a unique solution
of the form $\gamma_2(\sigma)=(u^{(2)}(\sigma),w_R,s)$ satisfying \[
  \lim_{\sigma\to+\infty}u^{(1)}(\sigma)= u_R,\quad
  \lim_{\sigma\to0^+}\left(
    u_2^{(2)},
    \frac{u_1^{(2)}}{\sqrt{u_2^{(2)}}}
  \right)(\sigma)
  = \left(+\infty,-\sqrt{3-\sqrt{3}}\right).
\]
\end{prop}

\begin{proof}
See \cite[Theorem 3.1]{Schaeffer:1993}.
\end{proof}

Motivated by Proposition \ref{prop_sss},
we compactify the state space by defining \beq{def_beta}
  \beta=\frac{u_1}{\sqrt{u_2}},\quad
  r= \frac1{\sqrt{u_2}}.
\] In this definition
we have assumed $u_2$ to be positive.
This is just for convenience
and has no loss of generality.
In general cases, since the value of $u_2$
is bounded from below along $\gamma_1$ and $\gamma_2$,
say $u_2>-M$,
we may replace $u_2$ by $u_2+M$.

In $(\beta,r,w_1,w_2,\xi,\epsilon)$-coordinates, \eqref{sf_u12} becomes,
after multiplying by $r$, 
\beq{sf_br}
  &\dot\beta
  = \tfrac{-1}{6}(\beta^4-6\beta^2+6)
  + r\left(
      \tfrac{-\beta\xi}2
      +r\left(\tfrac{\beta^2}2-w_1\right)
      +\tfrac{r^2}2\beta w_2
  \right)\\
  &\dot r= -\tfrac{\beta^3}6r
  +\tfrac{r^2}2\big(\xi+r\beta+r^2w_2\big)\\
  &\dot{w}_1= -\beta\epsilon\\
  &\dot{w}_2= \tfrac{-\epsilon}r\\
  &\dot{\xi}= r\epsilon\\
  &\dot{\epsilon}= 0.
\]
Note that the time scale in \eqref{sf_br} is different from that of \eqref{sf_u12},
but we use the same notation $\cdot$ to denote derivatives in time.
This should cause no ambiguity since
the time scales can be distinguished by the equations for $\dot\xi$.

In \eqref{sf_br}, the equation for $\dot{w}_2$ is not defined when $r=0$.
To make sense of it,
one naive way is to multiply the system by $r$,
but this will make the set $\{r=0\}$ non-normally hyperbolic.
To avoid this degeneracy,
our remedy is to replace $\epsilon$ by $\kappa= \epsilon/r$.
Then the system \eqref{sf_br} becomes
\beq{deq_brk}
  &\dot\beta
  = \tfrac{-1}{6}(\beta^4-6\beta^2+6)
  + r\left(
      \tfrac{-\beta\xi}2
      +r\left(\tfrac{\beta^2}2-w_1\right)
      +\tfrac{r^2}2\beta w_2
  \right)\\
  &\dot r= -\tfrac{\beta^3}6r
  +\tfrac{r^2}2\big(\xi+r\beta+r^2w_2\big)\\
  &\dot{w}_1= -\kappa \beta r\\
  &\dot{w}_2= -\kappa\\
  &\dot{\xi}= \kappa r^2\\
  &\dot\kappa
  = \tfrac{\beta^3}6\kappa
  + \tfrac{r}2\big(
    -\kappa \xi- r\beta\kappa- r^2\kappa w_2
  \big)
\]
Note that the first two equations in \eqref{sf_br} and \eqref{deq_brk} are identical.

The sets $\{u_2=+\infty\}$ and $\{\epsilon=0\}$
correspond to $\{r=0\}$ and $\{\kappa=0\}$.
Taking $r=0$ and $\kappa=0$,
the system \eqref{deq_brk} reduces to a single equation for $\beta$,
namely \beq{deq_beta}
  \dot\beta= \tfrac{-1}{6}(\beta^4-6\beta^2+6).
\] For this equation,
the equilibria are $\beta=\rho_j$, $j=1,\ldots,4$, where \beq{def_rho}
  \rho_1= -\sqrt{3+\sqrt3},\;
  \rho_2= -\sqrt{3-\sqrt3},\;
  \rho_3= \sqrt{3-\sqrt3},\;
  \rho_4= \sqrt{3+\sqrt3}.
\] Let \begin{align}
  &\mathcal P_L
  = \{(\beta,r,\kappa,w_1,w_2,\xi): \beta=\rho_3, r=0,\kappa=0\} \label{def_pl}\\
  &\mathcal P_R
  = \{(\beta,r,\kappa,w_1,w_2,\xi): \beta=\rho_2, r=0,\kappa=0\}. \label{def_pr}
\end{align} The trajectory $\gamma_1$ given in Proposition \ref{prop_sss}
connects $u_L$ and $\mathcal P_L$,
and $\gamma_2$
connects $u_R$ and $\mathcal P_R$.
Next we shall find connections
between $\mathcal P_L$ and $\mathcal P_R$.

We will find a trajectory on $\{r=0\}$
connecting $(\beta,\kappa,w_1,w_2,\xi)=(\rho_3,0,w_{1L},w_{2L},s)$
and $(\rho_2,0,w_{1R},w_{2R},s)$.
When $r=0$, the system
reduces to \begin{subequations}\label{deq_brk0}\begin{align}
  &\dot\beta
  = \tfrac{-1}{6}(\beta^4-6\beta^2+6)\label{deq_brk0_beta}\\
  &\dot\kappa
  = \tfrac{\beta^3}6\kappa\label{brk0_kappa}\\
  &\dot{w}_2= -\kappa\label{deq_brk0_w2}\\
  &\dot{w}_1= 0,\;
  \dot{\xi}=0.\label{deq_brk0_xi}
\end{align}\end{subequations}
Observe that the system \eqref{deq_brk0} is only weakly coupled,
so we can solve it by integration:

\begin{figure}[t]
\centering{
\includegraphics[trim = 2.4cm 7cm 1.4cm 7.4cm, clip, width=.66\textwidth]{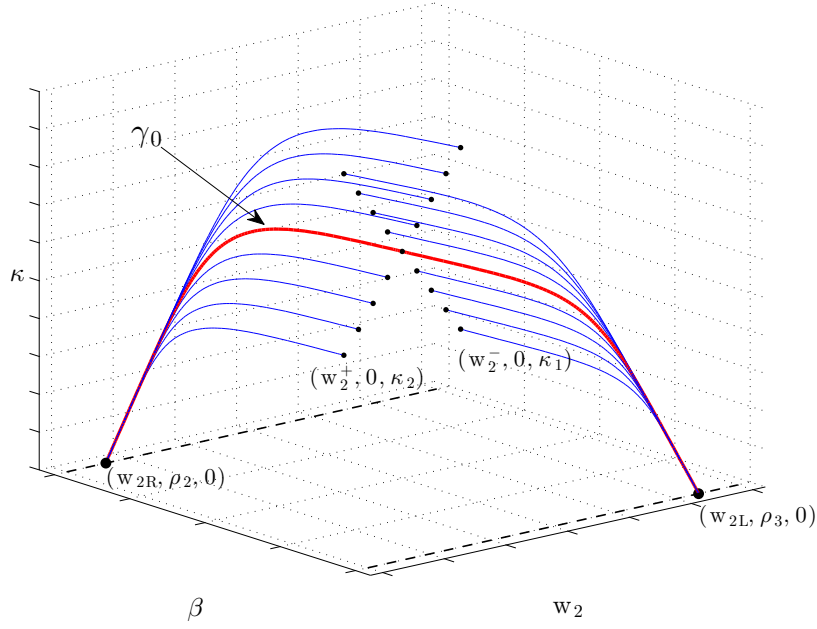}
}
\caption{
Trajectories for \eqref{deq_brk0}
starting at $(\rho_3,0,w_{1L},w_{2L},s)$,
and those ending at $(\rho_2,0,w_{1R},w_{2R},s)$.
}
\label{fig_gamma0}
\end{figure}

\begin{figure}[t]
\centering{
\includegraphics[trim = 2.4cm 7cm 1.4cm 7.4cm, clip, width=.66\textwidth]{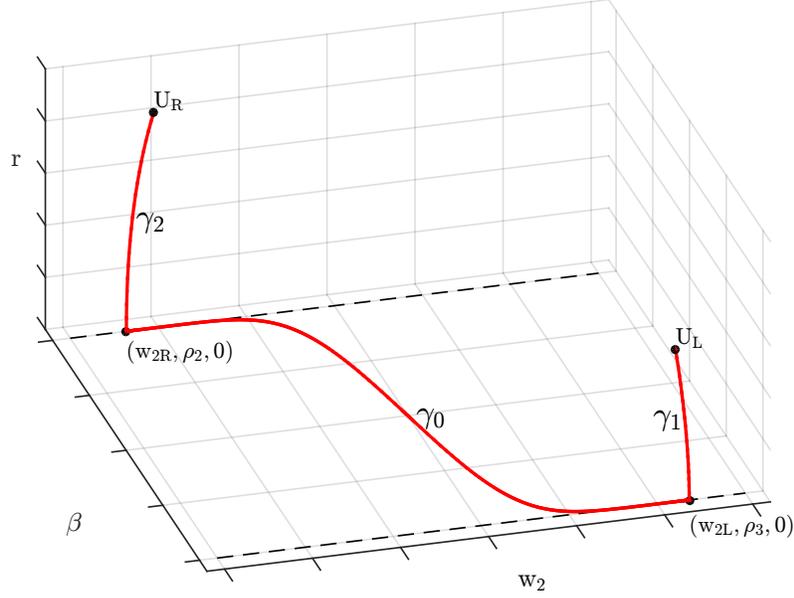}
}
\caption{
$\gamma_1$, $\gamma_2$ and $\gamma_0$
displayed in $(\beta,r,w_2)$-space.
}
\label{fig_gamma012}
\end{figure}

\begin{figure}[t]
\centering
\begin{parbox}{.68\textwidth}{\centering
{\includegraphics[trim = 2.2cm 7.2cm 1.2cm 7.2cm, clip, width=.66\textwidth]{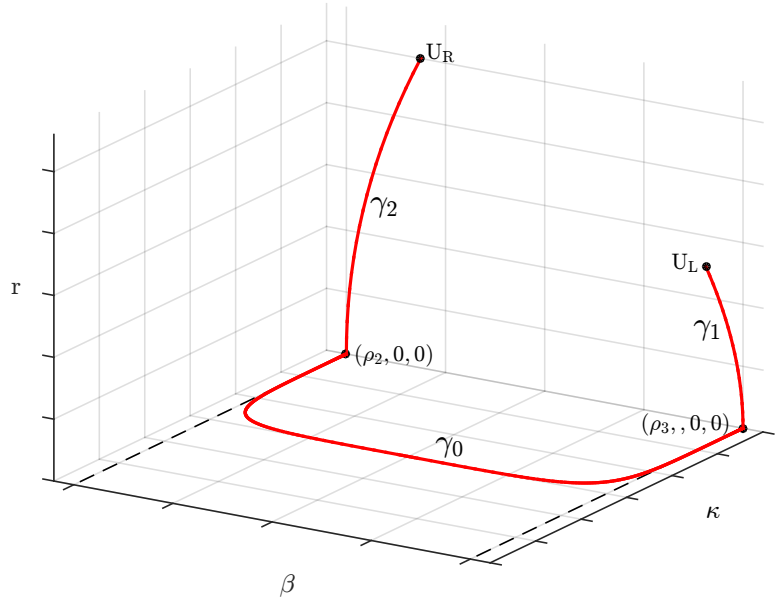}
}}\end{parbox}
\caption{
Near the singular configuration $\gamma_1\cup \gamma_0\cup \gamma_2$,
we will find trajectories for \eqref{deq_brk}
lying in the hyper-surface $\{rk=\epsilon\}$ for each small $\epsilon>0$.
}
\label{fig_gamma_kr}
\end{figure}

\begin{prop}\label{prop_wbk}
There exist positive smooth functions ${\iota_1}$, ${\iota_2}$ and ${\iota_3}$
which satisfy the following:
For any parameters $(\bar\kappa,\bar{w}_1,\bar{w}_2,\bar\xi)$,
the system \eqref{deq_brk0} with boundary conditions \beq{bc_kappa_minus}
  (\beta,\kappa)(0)=(0,\bar\kappa),\quad
  (w_1,w_2,\xi)(-\infty)= (\bar{w}_1,\bar{w}_2,\bar{\xi}),
\] has a unique solution \beq{def_kappa_minus}
  (\beta^-,\kappa^-,w_1^-,w_2^-,\xi^-)(\sigma)
  = \big(
    {\iota_1}(\sigma),\bar\kappa {\iota_2}(\sigma),\bar{w}_1,\bar{w}_2+\bar\kappa {\iota_3}(\sigma),\bar\xi
  \big).
\] For any parameters $(\hat\kappa,\hat{w}_1,\hat{w}_2,\hat\xi)$, 
the system \eqref{deq_brk0} with boundary conditions \beq{bc_kappa_plus}
  (\beta,\kappa)(0)=(0,\hat\kappa),\quad
  (w_1,w_2,\xi)(+\infty)= (\hat{w}_1,\hat{w}_2,\hat{\xi}),
\] has a unique solution \beq{def_kappa_plus}
  (\beta^+,\kappa^+,w_1^+,w_2^+,\xi^+)(\sigma)
  = \big(
    {\iota_1}(-\sigma),\hat\kappa {\iota_2}(-\sigma),
    \hat{w}_1,\hat{w}_2-\hat\kappa {\iota_3}(\sigma), \hat{\xi}
  \big).
\]
\end{prop}

\begin{proof}
First we solve \eqref{deq_brk0_beta}
by setting \beq{def_q1}
  \text{
    ${\iota_1}(\sigma)$ to be
    the solution of \eqref{deq_brk0_beta} satisfying ${\iota_1}(0)=0$.
  }
\] Let \beq{def_q2q3}
  {\iota_2}(\sigma)
  = \exp\left(\int_0^\sigma \frac{{\iota_1}(\tau)^3}6\;d\tau\right),\quad
  {\iota_3}(\sigma)
  = \int_{-\infty}^\sigma {\iota_2}(\tau)\;d\tau.
\]
Then a direct calculation shows that
\eqref{def_kappa_minus} and \eqref{def_kappa_plus}
are solutions of \eqref{deq_brk0} satisfying \eqref{bc_kappa_minus} and \eqref{bc_kappa_plus}.
\end{proof}

See Fig \ref{fig_gamma0} for the trajectories given in Proposition \ref{prop_wbk}.
Note that ${\iota_1}(\sigma)$ defined in \eqref{def_q1}
satisfies ${\iota_1}(-\infty)=\rho_3$ and ${\iota_2}(+\infty)=\rho_2$.

\begin{prop}\label{prop_gamma0}
If we set \beq{set_kappa_bar}
  (\bar\kappa,\bar{w}_1,\bar{w}_2,\bar\xi)
  = (\kappa_0,w_{1L},w_{2L},s),\quad
  (\tilde\kappa,\tilde{w}_1,\tilde{w}_2,\tilde\xi)
  = (\kappa_0,w_{1R},w_{2R},s),
\] with \beq{def_kappa0}
  \kappa_0= \frac{w_{2R}-w_{2L}}{2{\iota_3}(0)},
\] where ${\iota_3}(\sigma)$ is as defined in \eqref{def_q2q3},
then \eqref{def_kappa_minus} and \eqref{def_kappa_plus} coincide,
and this gives a solution of \eqref{deq_brk0},
denoted by $\gamma_0(\sigma)$,
satisfying \beq{bc_gamma0}
  \gamma_0(-\infty)
  = (\rho_3,0,w_{1L},w_{2L},s),\quad
  \gamma_0(+\infty)
  = (\rho_2,0,w_{1R},w_{2R},s).
\]
\end{prop}

\begin{proof}
First we set \[
  (\bar{w}_1,\bar{w}_2,\bar{\xi})
  = (w_{1L},w_{2L},s),\quad
  (\tilde{w}_1,\tilde{w}_2,\tilde{\xi})
  = (w_{1R},w_{2R},s).
\] From the definitions in \eqref{def_s} and \eqref{def_wLwR}
we have $w_{1L}=w_{1R}$.
Solving \[
  (\beta^-,\kappa^-,w_1^-,w_2^-,\xi^-)(0)
  = (\beta^+,\kappa^+,w_1^+,w_2^+,\xi^+)(0)
\] in \eqref{def_kappa_minus} and \eqref{def_kappa_plus}
for $\bar\kappa$ and $\tilde\kappa$,
we obtain the solution $\bar\kappa=\tilde\kappa=\kappa_0$
as defined in \eqref{def_kappa0}.
This gives a trajectory $\gamma_0(\sigma)$ satisfying \eqref{bc_gamma0}.
From the uniqueness of solutions of boundary value problems,
this trajectory is unique.
\end{proof}

We will show that the for the system \eqref{deq_brk}
there are trajectories close to $\gamma_1\cup \gamma_0\cup \gamma_2$
lying in hyper-surfaces $\{rk=\epsilon\}$, $\epsilon>0$.
See Fig \ref{fig_gamma012} and \ref{fig_gamma_kr}.

For solutions $(u_{1\epsilon},u_{2\epsilon})(\xi)$ of \eqref{dafermos_similar} and \eqref{bc_u_infty},
from the equation for $\dot{\xi}$ in \eqref{sf_u12},
we know the $\xi$-interval corresponding to any compact segment of $\gamma_1$ or $\gamma_2$
has length of order $O(\epsilon)$.
We will see at the end of Section \ref{sub{KK_sec_complete}istence} that
the length of the $\xi$-interval corresponding to any compact segment of $\gamma_0$
is of order $O(\epsilon^2)$.

\section{Geometric Singular Perturbation Theory}
\label{sec_GSPT}
Our main goal is to solve
the boundary value problem \eqref{dafermos_similar} and \eqref{bc_u_infty}.
Note that \eqref{dafermos_similar} is a \emph{singularly perturbed equation}
since the perturbation $\epsilon\frac{d^2}{d\xi^2}u$
has a higher order derivative than the other terms in the equation.
We will apply GSPT to deal with singularly perturbed equations.
The idea of GSPT
is to first study a set of subsystems which forms a decomposition of a system,
and then to use the information for the subsystems
to conclude results for the original system.

In Section \ref{subsec_fenichel} and \ref{subsec_silnikov},
we recall some fundamental theorems in GSPT.
We only briefly state necessary theorems
because it is similar to \cite[Section 4]{Hsu:2015a}.
In Section \ref{subsec_corner}
we state and give new proofs for a version
of the Corner Lemma.

\subsection{Fenichel's Theory for Fast-Slow Systems} \label{subsec_fenichel}
Note that \eqref{sf_u} is a \emph{fast-slow system},
which means that the system is of the form \beqeps{sf_xy}
  &\dot x=f(x,y,\epsilon)\\
  &\dot y=\epsilon g(x,y,\epsilon).
\] where $(x,y)\in\mathbb R^n\times \mathbb R^l$, 
and $\epsilon$ is a parameter.
In order to deal with fast-slow systems,
Fenichel's Theory
was developed in \cite{Fenichel:1973,Fenichel:1977,Fenichel:1979}.
Some expositions for that theory can be found in \cite{Wiggins:1994,Jones:1995}.

An important feature of a fast-slow system
is that the system can be decomposed into two subsystems:
the \emph{limiting fast system} and the \emph{limiting slow system}.
The limiting fast system is obtained by taking $\epsilon=0$ in \eqref{sf_xy};
that is, \beq{fast_xy}
  &\dot x=f(x,y,0)\\
  &\dot y=0.
\] 
On the other hand, note that the system \eqref{sf_xy} can be converted to,
after a rescaling of time,
\beqeps{sf_xy_singular}
  &\epsilon x'= f(x,y,\epsilon)\\
  &y'= g(x,y,\epsilon).
\] Taking $\epsilon=0$ in \eqref{sf_xy_singular}, 
we obtain the limiting slow system \beq{slow_xy}
  &0=f(x,y,0)\\
  &y'=g(x,y,0).
\] 
Note that the limiting slow system \eqref{slow_xy}
describes dynamics on the set of critical points of 
the limiting fast system \eqref{fast_xy},
so we will need to piece together the information of
the limiting fast system and the limiting slow system
in the vicinity of the set of critical points.
To piece this information together,
\emph{normal hyperbolicity} defined below will be a crucial condition.
\begin{defn}\label{defn_normally_hyperbolic}
A \emph{critical manifold} $\mathcal S_0$ for \eqref{fast_xy}
is an $l$-dimensional manifold consisting of critical points of \eqref{fast_xy}.
A critical manifold is \emph{normally hyperbolic}
if $D_xf(x,y,0)|_{\mathcal S_0}$ is hyperbolic.
That is,
at any point $(x_0,y_0)\in \mathcal S_0$,
all eigenvalues of $D_xf(x,y,0)|_{(x_0,y_0)}$
have nonzero real part.
\end{defn}

Fenichel's Theory is a center manifold theory for fast-slow systems.
For a normally hyperbolic critical manifold $\mathcal S_0$ for \eqref{fast_xy},
the stable and unstable manifolds
$W^s(\mathcal S_0)$ and $W^u(\mathcal S_0)$
can be defined in the natural way.
We denote them by
$W^s_0(\mathcal S_0)$ and $W^u_0(\mathcal S_0)$
to indicate their invariance under \eqref{sf_xy} with $\epsilon=0$.
Fenichel's Theory assures that
the hyperbolic structure of $\mathcal S_0$ persists under perturbation \eqref{sf_xy}.
Below we state three fundamental theorems
of Fenichel's Theory following \cite{Jones:1995}.

\begin{thm}[Fenichel's Theorem 1]
\label{thm_fenichel_invariant}
Consider the system \eqref{sf_xy},
where $(x,y)\in \mathbb R^n\times \mathbb R^l$,
and $f$, $g$ are $C^r$ for some $r\ge 2$.
Let $\mathcal S_0$ be a compact normally hyperbolic manifold for \eqref{fast_xy}.
Then for any small $\epsilon\ge 0$ there exist
locally invariant $C^r$ manifolds,
denoted by $\mathcal S_\epsilon$,
$W^s_\epsilon(\mathcal S_\epsilon)$
and $W^u_\epsilon(\mathcal S_\epsilon)$,
which are $C^1$ $O(\epsilon)$-close to
$\mathcal S_0$, $W^s_0(\mathcal S_0)$ and $W^u_0(\mathcal S_0)$, respectively.
Moreover, for any continuous families of compact sets
$\mathcal I_\epsilon\subset W^u_\epsilon(\mathcal S_\epsilon)$,
$\mathcal J_\epsilon\subset W^s_\epsilon(\mathcal S_\epsilon)$,
$\epsilon\in [0,\epsilon_0]$,
there exist positive constants $C$ and $\nu$
such that \begin{subequations}\label{est_dist}\begin{align}
  &\mathrm{dist}(z\cdot t,\mathcal S_\epsilon)
  \le Ce^{\nu t}
  \quad\forall\; z\in \mathcal I_\epsilon,\; t\le 0\\
  &\mathrm{dist}(z\cdot t,\mathcal S_\epsilon)
  \le Ce^{-\nu t}
  \quad\forall\; z\in \mathcal J_\epsilon,\; t\ge 0,
\end{align}\end{subequations}
where $\cdot$ denotes
the flow for \eqref{sf_xy}.
\end{thm}

\begin{proof}
See \cite[Theorem 3]{Jones:1995}.
\end{proof}

\begin{rmk}
If $\mathcal S_0$ is locally invariant under \eqref{sf_xy} for each $\epsilon$,
then the $\mathcal S_\epsilon$ can be chosen to be $\mathcal S_0$
because of the construction
in the proof of \cite[Theorem 3]{Jones:1995}.
\end{rmk}

Note that $W^u_\epsilon(\mathcal S_\epsilon)$ and $W^s_\epsilon(\mathcal S_\epsilon)$
can be interpreted as a decomposition
in a neighborhood of $\mathcal S_0$ in $(x,y)$-space.
The following theorem asserts that
this induces a change of coordinates $(a,b,c)$
such that $W^u_\epsilon(\mathcal S_\epsilon)$
and $W^s_\epsilon(\mathcal S_\epsilon)$
correspond to $(a,c)$-space and $(b,c)$-space, respectively.

\begin{thm}[Fenichel's Theorem 2]
\label{thm_fenichel_coordinate}
Suppose the assumptions in Theorem \ref{thm_fenichel_invariant} hold.
Then under a $C^r$ $\epsilon$-dependent
coordinate change $(x,y)\mapsto(a,b,c)$,
the system \eqref{sf_xy} can be brought to the form
\beqeps{sf_abc}
  &\dot a= A^u(a,b,c,\epsilon)a\\
  &\dot b= A^s(a,b,c,\epsilon)b\\
  &\dot c= \epsilon\big( h(c)+ E(a,b,c,\epsilon)\big)
\] in a neighborhood of $\mathcal S_\epsilon$,
where the coefficients are $C^{r-2}$ functions satisfying \beq{cond_spec}
  \inf_{\lambda\in \mathrm{Spec}A^u(a,b,c,0)}
  \mathrm{Re}\,\lambda>2\nu,\quad
  \sup_{\lambda\in \mathrm{Spec}A^s(a,b,c,0)}
  \mathrm{Re}\,\lambda<-2\nu
\] for some $\nu>0$ and \beq{cond_E}
  E=0\quad\text{on }\{a=0\}\cup \{b=0\}.
\] 
\end{thm}

\begin{proof}
See \cite[Section 3.5]{Jones:1995} or \cite[Proposition 1]{Jones:2009}.
\end{proof}

The family of trajectories for \eqref{slow_xy}
forms a foliation of $\mathcal S_0$.
The following theorem says that
this induces a foliation of
$W^u_\epsilon(\mathcal S_\epsilon)$ and $W^s_\epsilon(\mathcal S_\epsilon)$.

\begin{thm}[Fenichel's Theorem 3]
\label{thm_foliation}
Suppose the assumptions in Theorem \ref{thm_fenichel_invariant} hold.
Let $\Lambda_0$ be a submanifold in $\mathcal S_0$
which is locally invariant under \eqref{slow_xy}.
Then there exist locally invariant manifolds $\Lambda_\epsilon$,
$W^s_\epsilon(\Lambda_\epsilon)$,
and $W^u_\epsilon(\Lambda_\epsilon)$
for \eqref{sf_xy}
which are $C^{r-2}$ $O(\epsilon)$-close to
$\Lambda_0$,
$W^s_0(\Lambda_0)$,
and $W^u_0(\Lambda_0)$, respectively.
Moreover, for any continuous families of compact sets
$\mathcal I_\epsilon\subset W^u_\epsilon(\Lambda_\epsilon)$,
$\mathcal J_\epsilon\subset W^s_\epsilon(\Lambda_\epsilon)$,
$\epsilon\in [0,\epsilon_0]$,
there exist positive constants $C$ and $\nu$
such that \eqref{est_dist}
holds with $\mathcal S_\epsilon$ replaced by $\Lambda_\epsilon$.
Suppose in addition that
$S_0$ is invariant under \eqref{sf_xy} for each $\epsilon$.
Then $\Lambda_\epsilon$ can be chosen to be $\Lambda_0$.
\end{thm}

\begin{proof}
Using Fenichel's coordinates $(a,b,c)$
in Theorem \ref{thm_fenichel_coordinate} for the splitting of $\mathcal S_0$,
we can take $W_\epsilon^u(\Lambda_\epsilon)$ and $W_\epsilon^s(\Lambda_\epsilon)$
to be the pre-images of the sets $\{(a,b,c): a=0, c\in \Lambda_0\}$
and $\{(a,b,c): b=0, c\in \Lambda_0\}$, respectively,
in $(x,y)$-space.
From \eqref{cond_spec} we obtain \eqref{est_dist}
with $\mathcal S_\epsilon$ replaced by $\Lambda_\epsilon$.
Suppose $S_0$ is invariant under \eqref{sf_xy} for each $\epsilon$,
then from the remark after Theorem \ref{thm_fenichel_invariant},
we can take $\mathcal S_\epsilon=\mathcal S_0$
and hence $\Lambda_\epsilon=\Lambda_0$
\end{proof}

The system \eqref{sf_abc} is called a \emph{Fenichel normal form} for \eqref{sf_xy},
and the variables $(a,b,c)$ are called \emph{Fenichel coordinates}.

\subsection{Silnikov Boundary Value Problem} \label{subsec_silnikov}

We have seen in Section \ref{subsec_fenichel}
that fast-slow systems \eqref{sf_xy}
can locally be converted into normal forms \eqref{sf_abc},
where $A^u$ and $A^s$ satisfy the gap condition \eqref{cond_spec},
and $E$ is a small term satisfying \eqref{cond_E}.
If we append the system with the equation $\dot\epsilon=0$
and then replace $c$ by $\tilde c=(c,\epsilon)$,
we obtain a system of the form \beq{deq_abc_center}
  &\dot{a}= A^u(a,b,\tilde c)a\\
  &\dot{b}= A^s(a,b,\tilde c)b\\
  &\dot{\tilde c}= \tilde h(\tilde c)+ E(a,b,\tilde c),
\] for which \eqref{cond_spec} and \eqref{cond_E} are satisfied
with $E$ replaced by $\tilde E$.
For convenience, we will drop the tilde notation in \eqref{deq_abc_center}
in the remaining discussion.

A Silnikov problem is the system \eqref{deq_abc_center}
along with boundary data of the form \beq{bc_silnikov}
  (b,c)(0)=(b^0,c^0),\quad
  a(T)=a^1,
\] where $T\ge 0$.

The critical manifold for \eqref{deq_abc_center} is $\{a=0,b=0\}$,
on which the system is governed by
the limiting slow system \beq{deq_c_critical}
  \dot{c}= h(c).
\] For a solution $(a(t),b(t),c(t))$ to
the Silnikov boundary value problem \eqref{deq_abc_center} and \eqref{bc_silnikov},
from conditions \eqref{cond_spec} and \eqref{cond_E},
it is natural to expect that
$a(t)$ and $b(t)$ decay to $0$ in backward time and forward time, respectively,
and that $c(t)$ is approximately the solution of \eqref{deq_c_critical}.
A theorem from \cite{Schecter:2008a}
asserts that this is the case:

\begin{thm}[Generalized Deng's Lemma \cite{Schecter:2008a}]
\label{thm_bvp_h}
Consider the system \eqref{deq_abc_center}
satisfying \eqref{cond_spec} and \eqref{cond_E}
with $C^r$ coefficients, $r\ge 1$,
defined on the closure of a bounded open set
$B_{k,\Delta}\times B_{m,\Delta}\times V
\subset \mathbb R^k\times \mathbb R^m\times \mathbb R^l$,
where $B_{k,\Delta}=\{a\in\mathbb R^k: |a|<\Delta\}$,
$\Delta>0$, and $V$ is a bounded open set in $\mathbb R^l$.

Let $K_0$ and $K_1$ be compact subsets of $V$
such that $K_0\subset \mathrm{Int}(K_1)$.
For each $c^0\in K_0$ let $J_{c^0}$ be the maximal interval
such that $\phi(t,c^0)\in \mathrm{Int}(K_1)$ for all $t\in J_{c^0}$,
where $\phi(t,c^0)$ is the solution of \eqref{deq_c_critical} with initial value $c^0$.
Let $\nu>0$ be the number in \eqref{cond_spec}.
Suppose there exists $\beta>0$ such that $\tilde \nu:= \nu- r\beta>0$ and \[
  |\phi(t,c^0)|\le Me^{\beta|t|}\quad\forall\; t\in J_{c^0}.
\] Then there is a number $\delta_0>0$
such that if $|a^1|<\delta_0$, $|b^0|<\delta_0$, $c^0\in V_0$, and ${T}>0$ is in $J_{c^0}$,
then the Silnikov boundary value problem \eqref{deq_abc_center} and \eqref{bc_silnikov}
has a solution $(a,b,c)(t,{T},a^1,b^0,c^0)$ on the interval $0\le t\le {T}$.
Moreover, there is a number $K>0$ such that for all $(t,{T},a^1,b^0,c^0)$ as above
and for all multi-indices $\mathbf{i}$ with $|\mathbf{i}|\le r$, \beq{est_Di_bvp}
  &|D_{\mathbf{i}}a(t,{T},a^1,b^0,c^0)|\le Ke^{-\tilde\nu ({T}-t)}\\
  &|D_{\mathbf{i}}b(t,{T},a^1,b^0,c^0)|\le Ke^{-\tilde\nu t}\\
  &|D_{\mathbf{i}}c(t,{T},a^1,b^0,c^0)-D_{\mathbf{i}}\phi(t,c^0)|
  \le Ke^{-\tilde\nu T}.
\]
\end{thm}

\subsection{The Corner Lemma}
\label{subsec_corner}
\begin{figure}[t]
\centering
\begin{parbox}{.48\textwidth}{\centering
\boxed{\includegraphics[trim = 5cm 7.5cm 4cm 7.4cm, clip, width=.47\textwidth]{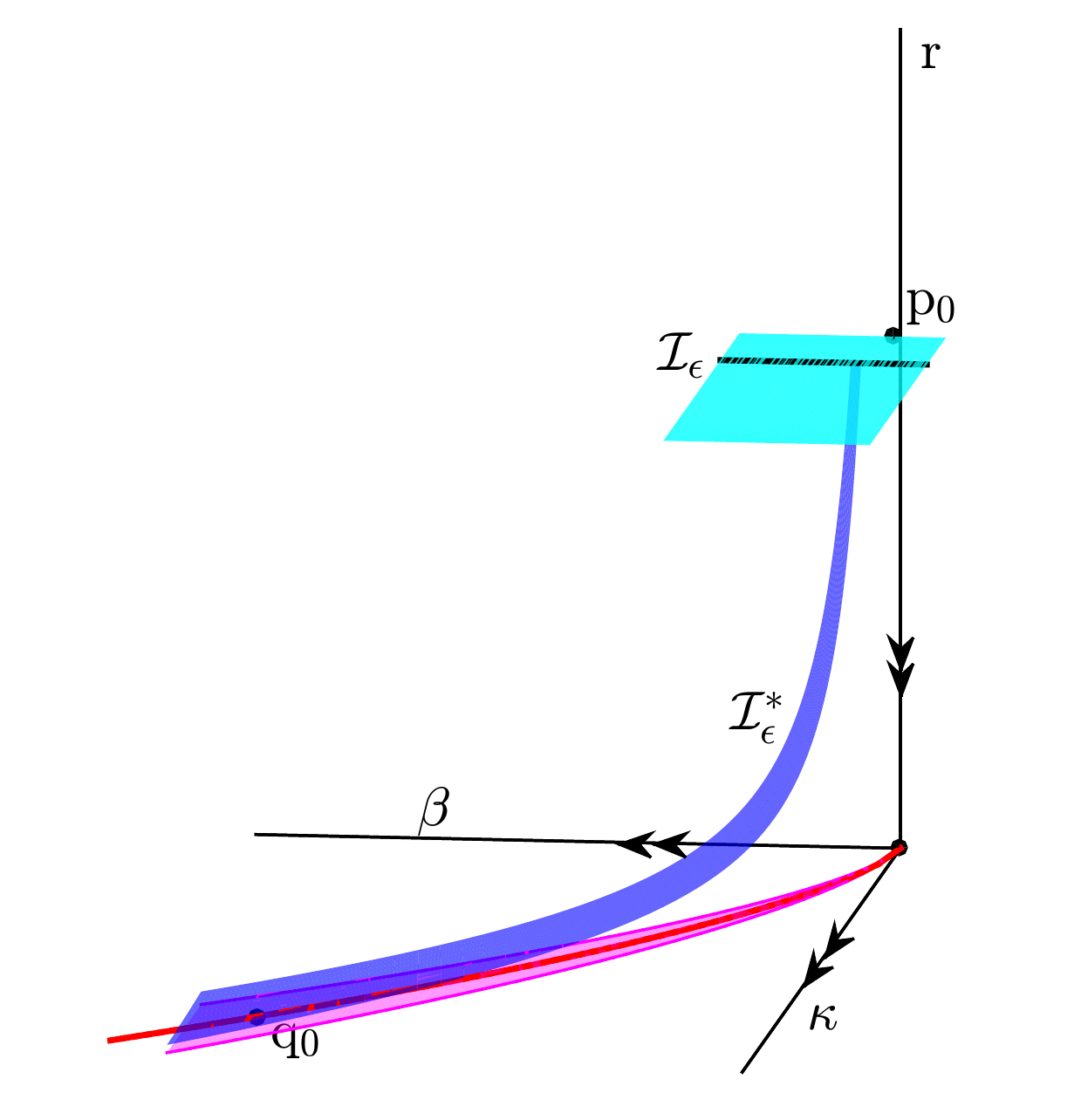}
}\\(a)}\end{parbox}\;
\begin{parbox}{.48\textwidth}{\centering
\boxed{
\includegraphics[trim = 6.5cm 7.5cm 2.5cm 7.4cm, clip, width=.47\textwidth]{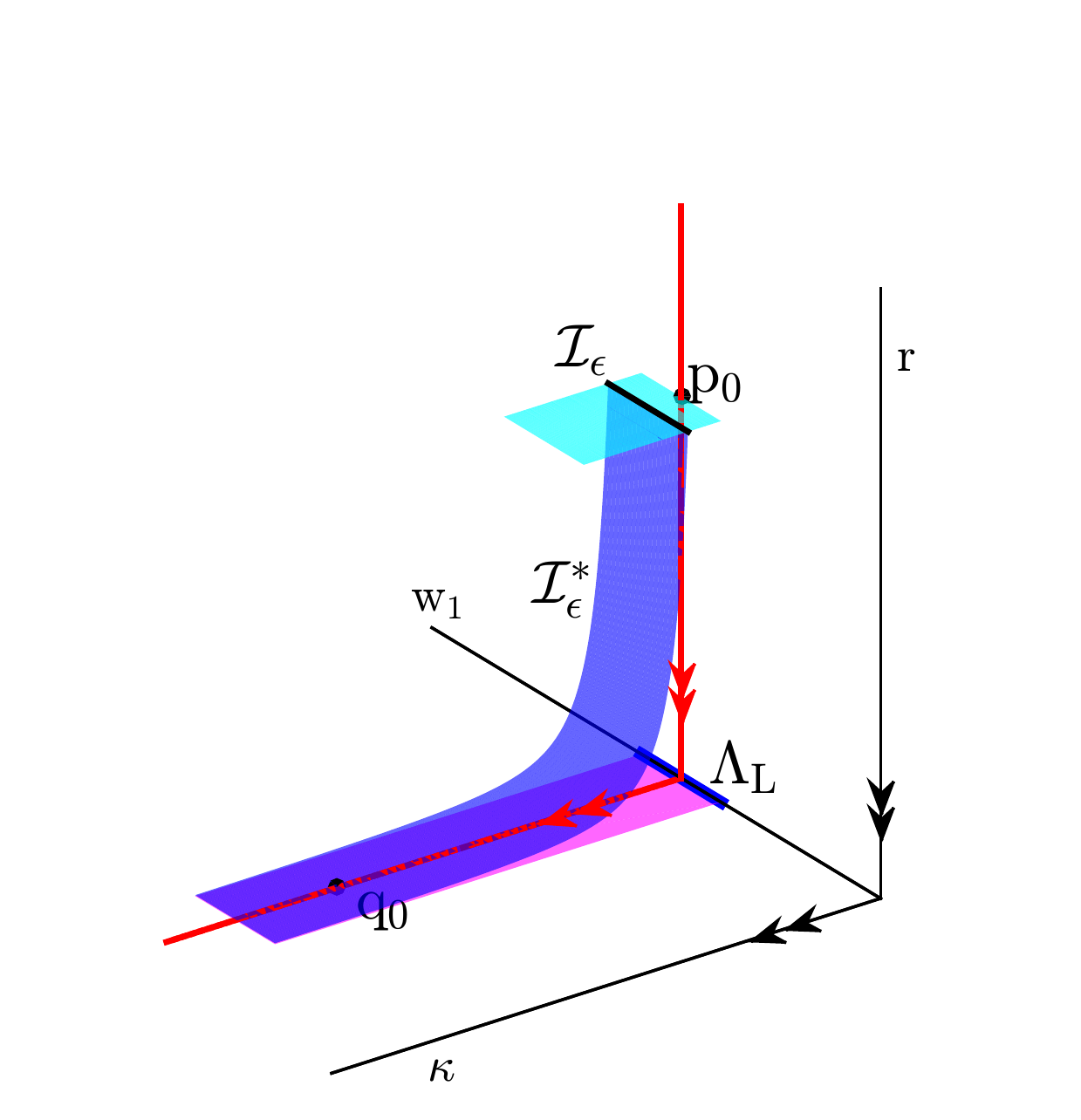}
}\\ (b)}\end{parbox}
\caption{The Corner Lemma with $a=(\beta,\kappa)$, $b=r$ and $c=(w_1,w_2,\xi)$.
The projections of the dynamics  for this $6$D system
in $(\beta,\kappa,r)$- and $(\kappa,r,w_1)$-spaces
are illustrated in (a) and (b), respectively.
Note that $\mathcal I_\epsilon^*$ expands exponentially in the $\beta$-direction,
but in the $w_1$-direction it changes only mildly.}
\label{fig_corner}
\end{figure}

The Corner Lemma was first asserted in \cite{Schecter:2004},
but its author later pointed out \cite[Remark 2.4]{Schecter:2008}
that the proof was flawed and needed to be reworked.
In Theorem \ref{thm_corner} we
modify both the statement and the proof of the original lemma.
In our modified version,
the required assumptions are more restricted,
but they are already enough for our purpose. 

First we state 
the special case of Theorem \ref{thm_bvp_h}
with $h\equiv 0$ in \eqref{deq_abc_center} as follows.
\begin{thm}\label{thm_bvp_h0}
Consider a system of the form \beq{deq_abc}
  &\dot{a}= A^u(a,b,c)a\\
  &\dot{b}= A^s(a,b,c)b\\
  &\dot{c}= E(a,b,c),
\] satisfying \eqref{cond_spec} and \eqref{cond_E}
with $C^r$ coefficients, $r\ge 1$,
defined on the closure of a bounded open set 
$\mathcal B= B_{k,\Delta}\times B_{m,\Delta}\times V
\subset \mathbb R^k\times \mathbb R^m\times \mathbb R^l$.
Then for any $(a^1,b^0,c^0)\in \mathcal B$ and $T\ge 0$,
the Silnikov boundary value problem \eqref{deq_abc} and \eqref{bc_silnikov} has a unique solution,
denoted by $(a,b,c)(t;T,a^1,b^0,c^0)$, $t\in [0,T]$.
Moreover, if we set \[
  p_T= (a,b,c)(0;T,a^1,b^0,c^0),\quad
  q_T= (a,b,c)(T;T,a^1,b^0,c^0)
\] and write $p_T=(a^\din_T,b^0,c^0)$ and $q_T=(a^1,\hat{b}_T,\hat{c}_T)$,
then \beq{est_pin_bvp}
  \|(a^\din_T,\hat{b}_T,\hat{c}_T-c^0)\|_{C^r(\mathcal B)}
  \le \tilde C e^{-\mu T}
\] for some positive constants $\tilde C$ and $\mu$.
\end{thm}

We will consider special cases of the system \eqref{deq_abc}
for which there is an invariant manifold of codimension $1$
which is transverse to an unstable direction.
For definiteness, we assume $\{a_k=0\}$ to be invariant under \eqref{deq_abc},
and the matrix-valued function $A^u(a,b,c)$ is of the form \beq{def_Au0}
  A^u= \begin{pmatrix}
    A_0^u& *\\
    0& \lambda_k
  \end{pmatrix}
\] where $A_0^u$ is a $(k-1)\times(k-1)$ matrix function
and $\lambda_k$ is a positive scalar function.

Using Theorem \ref{thm_bvp_h0},
we will prove the following:

\begin{thm}[Corner Lemma]\label{thm_corner}
Consider \eqref{deq_abc} defined
on the closure of
a bounded open set
$B_{k,\Delta}\times B_{m,\Delta}\times V
\subset \mathbb R^k\times \mathbb R^m\times \mathbb R^l$,
where the coefficients $A^u$, $A^s$ and $E$ are $C^r$ for some $r\ge 3$,
and $A^u$ is of the form \eqref{def_Au0}.
Assume \eqref{cond_spec} and \beq{cond_E_weak}
  E(a,b,c)=0
  \quad\text{on }\;
  \{a=0\}\cap \{b=0\}.
\] Let $\Lambda\subset V$ be a $\sigma$-dimensional $C^r$ manifold,
$0\le \sigma\le l$,
and let $\mathcal I$ be a $C^r$ manifold of the form \beq{def_Icorner}
  \mathcal I
  = \{(a,b,c): |a|<\Delta_1, b=b^0, c=c^0+ \theta(a,c^0)a, c^0\in \Lambda\},
\] where $0<\{\Delta_1,|b^0|\}<\Delta$,
and $\theta$ is a $(l\times k)$-matrix function.
Let $\mathcal I_\epsilon= \mathcal I\cap \{a_k=\epsilon\}$.
Denote ${\mathcal I}_\epsilon^*= {\mathcal I}_\epsilon\cdot [0,\infty)$.
Then the following holds:
Fix any $q_0\in W^u(\Lambda)$ with positive $a_k$-coordinate.
Then there exists a neighborhood $V_0$ of $q_0$
satisfying that \beq{Ieps_V0_corner}
  \text{
    $\mathcal I_\epsilon^*\cap V_0$
    is $C^{r-3}$ close to $W^u(\Lambda)\cap V_0$
  }
\] as $\epsilon\to 0$. See Fig \ref{fig_corner}.

Furthermore,
given any sequence of points $q_\epsilon\in \mathcal I_\epsilon^*\cap V_0$,
$\epsilon\in [0,\epsilon_0]$,
which converges to a point $q_0\in W^u(\Lambda_L)$ as $\epsilon\to 0$,
let $p_\epsilon\in \mathcal I_\epsilon$ and $T_\epsilon>0$
be such that $q_\epsilon= p_\epsilon\cdot T_\epsilon$,
and let $p_0$ be the unique point in $\mathcal I_0$
satisfying $\pi^s(p_0)=\pi^u(q_0)$,
where $\pi^{s,u}$ are the projections along stable/unstable fibers.
Then $p_\epsilon\to p_0$ as $\epsilon\to 0$, 
and \beq{est_Teps_corner}
  \tilde{C}^{-1}\log\frac1\epsilon
  \le T_\epsilon
  \le \tilde{C}\log\frac1\epsilon
\] for some $\tilde{C}>0$.
\end{thm}

\begin{proof}
Under the assumption \eqref{cond_E_weak},
from \cite[Lemma 2.2]{Deng:1990},
there exists a $C^{r-2}$ change of variables
of the form $(a,b,c)\mapsto (a,b,\hat{c})$
so that the new system converted from \eqref{deq_abc},
still denoted by \eqref{deq_abc},
satisfies \eqref{cond_E}.
The change of coordinate is a modification only on $c$,
so $\mathcal I$ is still parametrized as \eqref{def_Icorner}
in the new coordinates.
Therefore, by dropping the hat in $\hat{c}$,
we assume \eqref{cond_E} holds for the system \eqref{deq_abc},
and the coefficients are $C^{r-2}$ functions.

The stable/unstable manifolds for \eqref{deq_abc} are \beq{WuLambda_corner}
  W^s(\Lambda)= \{(a,b,c): b=0\},\quad
  W^u(\Lambda)= \{(a,b,c): a=0\}.
\] From \eqref{cond_E}, the slow variable $c$ is constant on $\{a=0\}\cup\{b=0\}$,
which implies \beq{pi_corner_pf}
  \pi^u(a,0,c)= (0,0,c),\quad
  \pi^s(0,b,c)= (0,0,c).
\] 

Let \beq{def_A_corner}
  \mathcal A=\{a\in \mathbb R^k: |a-a(q_0)|<\Delta_2\}
\] for some positive number $\Delta_2<\frac12\min\{\Delta,|a(q_0)|,a_k(q_0)\}$,
so that $\mathcal A\subset B_{k,\Delta}$,
where $a(q_0)$ and $a_k(q_0)$ denote the $a$- and $a_k$-coordinates of $q_0$.
Choose a smooth real-valued function $\chi(b)$
so that $\chi(b^0)=1$ and $\chi(0)=0$.
Let \beq{def_ctilde}
  &\tilde{c}
  = c- \chi(b)\theta(a,c^0)a.
\] Then from $\chi(b^0)=0$ we have \begin{align}
  &\tilde c
  = c- \theta(a,c^0)a
  \quad\text{on }\; \{b=b^0\}  \label{ctilde_b0}
\intertext{and from $\chi(0)=0$ we have}
  &\tilde c
  = c
  \quad\text{on }\; \{a=0\}\cup \{b=0\}.  \label{ctilde_a0}
\end{align}
From \eqref{ctilde_b0},
the image of $\mathcal I$ in $(a,b,\tilde{c})$-space is \beq{Itilde_para}
  \tilde{\mathcal I}
  = \{(a,b,\tilde{c}): |a|<\Delta_1, b=b^0, \tilde{c}=c^0, c^0\in \Lambda\}.
\] 
From \eqref{cond_spec} we know \beq{est_lambdak}
  \tilde{C}^{-1}< \lambda_k< \tilde{C}
\] for some positive constant $\tilde{C}$.
In $(a,b,\tilde c)$-coordinates,
the system \eqref{deq_abc} is converted to,
after dividing the equation by $\lambda_k$, \beq{deq_abc_tilde}
  &{a}'=\begin{pmatrix}
    \tilde{A}^u_0&*\\  0&1
  \end{pmatrix}a,\quad
   {b}'= \tilde{A}^s\, b,\quad
  {\tilde{c}}'= \tilde{E},
\] for some $C^{r-3}$ coefficients $\tilde{A}^u_0$, $\tilde{A}^s$ and $\tilde{E}$,
where $\prime$ denotes the derivative
with respect to the time variable $\zeta$ defined by \beq{deta_dt}
  d\zeta/d\sigma= \lambda_k,
\] where $\sigma$ is the time variable for \eqref{deq_abc}.
Clearly \eqref{cond_spec} holds
with $A^u$, $A^s$ and $\nu$ replaced by $\tilde{A}^u$, $\tilde{A}^s$
and $\tilde{\nu}:=\nu/\tilde{C}$.
Note that the condition \eqref{cond_E}
means $c$ is constant on $\{a=0\}\cup \{b=0\}$.
From \eqref{ctilde_a0}
we see that
$\tilde{c}$ is also constant on that set.
Hence \eqref{cond_E} holds with $E$ replaced by $\tilde E$.
Thus Theorem \ref{thm_bvp_h0} can be applied to \eqref{deq_abc_tilde}.

By Theorem \ref{thm_bvp_h0},
for any sufficiently large number $T$ and any $(a^1,c^0)\in\mathcal A\times \Lambda$, 
we can set $(a,b,\tilde{c})(t ;T,a^1,b^0,c^0)$, $t\in [0,T]$,
to be the solution of \eqref{deq_abc_tilde} satisfying 
\beq{bc_corner_pf}
  (b,\tilde c)(0)= (b^0,c^0),\quad
  a(T)= a^1.
\] Since the equation for $a_k$ in \eqref{deq_abc_tilde} is $a_k'=a_k$,
by choosing $T=\zeta_\epsilon:=\log(a_k^1/\epsilon)$,
where $a_k^1$ is the $a_k$-coordinate of $a^1$,
the solution corresponding to \eqref{bc_corner_pf} satisfies $a_k(0)=\epsilon$.
We set \[
  \tilde{p}_\epsilon
  = (a,b,\tilde{c})(0;\zeta_\epsilon,a^1,b^0,c^0),\quad
  \tilde{q}_\epsilon
  = (a,b,\tilde{c})(\zeta_\epsilon;\zeta_\epsilon,a^1,b^0,c^0),
\] and let $p_\epsilon$ and $q_\epsilon$
be the images of $\tilde{p}_\epsilon$ and $\tilde{q}_\epsilon$, respectively,
in $(a,b,c)$-space.
From \eqref{Itilde_para} we see that $\tilde{p}_\epsilon\in \tilde{\mathcal I}$,
and hence $p_\epsilon\in {\mathcal I}$.
Since the $a_k$-coordinate of $p_\epsilon$ is $a_k(0)=\epsilon$,
we conclude that $p_\epsilon\in {\mathcal I}_\epsilon$.

Regarding $\tilde{p}_\epsilon$ and $\tilde{q}_\epsilon$
as functions of $(a^1,c^0)\in \mathcal A\times \Lambda$,
using \eqref{est_pin_bvp}
with $T$ and $\nu$ replaced by $\zeta_\epsilon$ and $\tilde{\nu}$,
we have \beq{est_pintilde_corner}
  \|\tilde{p}_\epsilon- (0,b^0,c^0)\|_{C^{r-3}(\mathcal A\times\Lambda)}
  + \|\tilde{q}_\epsilon- (a^1,0,c^0)\|_{C^{r-3}(\mathcal A\times\Lambda)}
  \le C\epsilon^{\tilde\nu}.
\] From \eqref{ctilde_a0} it follows that
the $\tilde c$-coordinates of $\tilde{p}_\epsilon$ and $\tilde{q}_\epsilon$
are $O(\epsilon^{\tilde\nu})$-close to $c^0$ in $C^{r-2}$-norm.
Hence \eqref{est_pintilde_corner} holds
with $\tilde{p}_\epsilon$ and $\tilde{q}_\epsilon$
replaced by $p_\epsilon$ and $q_\epsilon$.
Since $p_\epsilon$ and $q_\epsilon$
parametrize $\mathcal I_\epsilon$ and $\mathcal I_\epsilon^*$
in neighborhoods of $p_0$ and $q_0$,
by \eqref{WuLambda_corner} this proves \eqref{Ieps_V0_corner}.

Next we consider 
the sequences $q_\epsilon$ and $p_\epsilon$ described in the statement.
Write \[
  p_\epsilon= (a^\din_\epsilon,b^\din_\epsilon,c^\din_\epsilon),\quad
  q_\epsilon= (\hat{a}_\epsilon,\hat{b}_\epsilon,\hat{c}_\epsilon),
\] and $q_0= (a^1,0,c^0)$ in $(a,b,c)$-coordinates.
By the definition of $\mathcal I$, we have $b^\din_\epsilon=b^0$.
The assumption $q_\epsilon\to q_0$ gives $\hat{c}_\epsilon\to c^0$,
and then by \eqref{pi_corner_pf}
the assumption $\pi^u(q_0)=\pi^s(p_0)$ implies $p_0= (0,b^0,c^0)$.
From \eqref{est_pintilde_corner}
we have $a^\din_\epsilon=o(1)$ and $c^\din_\epsilon= \hat{c}_\epsilon+o(1)$.
It follows that $c^\din_\epsilon\to c^0$,
and hence $p_\epsilon\to p_0$.

Let $T_\epsilon>0$ be the number such that
$q_\epsilon= p_\epsilon\cdot T_\epsilon$.
Since $p\in\mathcal I_\epsilon$,
the $a_k$-coordinate of $p_\epsilon$ equals $\epsilon$,
so from \eqref{deta_dt} we have \beq{Teps_lambda}
  T_\epsilon
  = \int_0^{\zeta_\epsilon}\frac{1}{\lambda_k}\;d\zeta,\quad\text{where }\;
  \zeta_\epsilon= \log\frac{a_k(q_\epsilon)}{\epsilon}
  = \log\frac{a_k(q_0)+o(1)}{\epsilon}.
\] Inserting \eqref{est_lambdak} in \eqref{Teps_lambda},
we then obtain \eqref{est_Teps_corner}.
\end{proof}

\section{Singular Configuration}
\label{sec_singular_config}
We will find trajectories of limiting subsystems
of the fast-slow system \eqref{sf_u12}
such that the union of those trajectories
forms a singular configuration joining the end states $u_L$ and $u_R$.

\subsection{End States $\mathcal U_L$ and $\mathcal U_R$} \label{subsec_fenichel_UL}
Observe that the system \eqref{fast_u12}
has a normally hyperbolic critical manifold \beq{def_s0}
  \mathcal S_0
  = \big\{
    (u,w,\xi):
    f(u)- \xi u - w= 0,
    \xi\ne \mathrm{Re }(\lambda_\pm(u))
  \big\},
\] where $\lambda_{\pm}(u)$ are the eigenvalues of $Df(u)$,
as defined in \eqref{def_lambda_pm}.
The limiting slow system for \eqref{sf_u12} is \beq{slow_u}
  &0=f(u)-\xi u- w\\
  &w'=-u\\
  &\xi'=1.
\]
From $\mathrm{(H1)}$
we have $s<\mathrm{Re}(\lambda_\pm(u_L))$,
so $(u_L,w_L,s)\in \mathcal S_0$.
Choose $\delta>0$ so that $s+2\delta<\mathrm{Re}(\lambda_\pm(u_L))$,
and set \beq{def_ul}
  \mathcal U_L
  &= (u_L,w_L,s)\myflow{slow_u} (-\infty,\delta]\\
  &= \{(u,w,\xi): u=u_L, w=w_L-\alpha_1 u_L, \xi= s+\alpha_1, \alpha_1\in(-\infty,\delta]\},
\] where $\myflow{slow_u}$
denotes the flow for \eqref{slow_u}.
It is clear that $\mathcal U_L\subset \mathcal S_0$
is normally hyperbolic with respect to \eqref{fast_u12},
and is locally invariant with respect to \eqref{sf_u}.

Note that each point in $\mathcal U_L$
is a hyperbolic equilibrium for the $2$-dimensional system \eqref{fast_u},
and the unstable manifold $W^u_0(\mathcal U_L)$ is naturally defined.

\begin{prop}\label{prop_ul}
Assume $\mathrm{(H1)}$.
Let $\mathcal U_L$ be defined in \eqref{def_ul}.
Fix any $r\ge 1$.
There exists a family of invariant manifolds
$W_\epsilon^u(\mathcal U_L)$
which are $C^k$ $O(\epsilon)$-close to $W_0^u(\mathcal U_L)$
such that for any continuous family $\{\mathcal I_\epsilon\}_{\epsilon\in [0,\epsilon_0]}$ of
compact sets $\mathcal I_\epsilon\subset W_\epsilon^u(\mathcal U_L)$, \beq{est_dist_ul}
  \mathrm{dist}(p\myflow{sf_u} t,\mathcal U_L)
  \le Ce^{\mu t}
  \quad\forall\; p\in \mathcal I_\epsilon,\; t\le 0, \epsilon\in [0,\epsilon_0],
\] for some positive constants $C$ and $\mu$.
\end{prop}

\begin{proof}
This follows
from Theorem \ref{thm_foliation}
by taking $\mathcal U_L$ to be $\mathcal{U}_0$.
Although $\mathcal U_L$ is not compact,
it is uniformly normally hyperbolic 
since $\xi-\mathrm{Re}(\lambda_\pm(u_L))<-\delta$ on $\mathcal U_L$,
and the proof of Theorem \ref{thm_foliation}
in \cite[Theorem 4]{Jones:1995} is still valid.
\end{proof}

\begin{rmk}
Proposition \ref{prop_ul}
was also asserted in \cite{Schecter:2004,Liu:2004,Keyfitz:2012}.
\end{rmk}

From $\mathrm{(H1)}$ we also have,
by decreasing $\delta$ if necessary,
$s-2\delta>\mathrm{Re}(\lambda_\pm(u_L))$,
and hence a similar result holds for for the set $\mathcal U_R$ defined by
\beq{def_ur}
  \mathcal U_R
  &= (u_R,w_R,s)\myflow{slow_u} [-\delta,\infty)\\
  &= \{(u,w,\xi): u=u_R, w=w_R-\alpha_2 u_R, \xi= s+ \alpha_2, \alpha_2\in [-\delta,\infty)\}.
\]

\begin{prop}
Assume $\mathrm{(H1)}$.
Let $\mathcal U_R$ be defined by \eqref{def_ur}.
Fix any $k\ge 1$.
There exists a family of invariant manifolds
$W_\epsilon^s(\mathcal U_R)$
which are $C^k$ $O(\epsilon)$-close to $W_0^s(\mathcal U_R)$
such that for any continuous family
$\{\mathcal J_\epsilon\}_{\epsilon\in [0,\epsilon_0]}$ 
of compact sets
$\mathcal J_\epsilon\subset W^s_\epsilon(\mathcal U_R)$,
\beq{est_dist_ur}
  \mathrm{dist}(p\myflow{sf_u} t,\mathcal U_R)
  \le Ce^{-\mu t}
  \quad\forall\; p\in \mathcal J_\epsilon,\; t\ge 0, \epsilon\in [0,\epsilon_0],
\] for some positive constants $C$ and $\mu$.
\end{prop}

\subsection{Intermediate States $\mathcal P_L$ and $\mathcal P_R$}
\label{subsec_hyp_UL}

It is easy to see that $\mathcal P_L$ defined in \eqref{def_pl}
is a normally hyperbolic critical manifold for \eqref{deq_brk},
so $C^k$ unstable and stable manifolds
$W^u(\mathcal P_L)$ and $W^s(\mathcal P_L)$ of $\mathcal P_L$ exist
for any fixed $k\ge 1$.
Note that $\{r=0\}$ and $\{\kappa=0\}$ are invariant under \eqref{deq_brk}
while $\{\beta=\rho_3\}$ is not.
We can straighten $W^u(\mathcal P_L)$ and $W^s(\mathcal P_L)$ 
by modifying $\beta$:

\begin{prop}\label{prop_pl}
Let $W^{u,s}(\mathcal P_L)$
be $C^k$ unstable/stable manifolds of $\mathcal P_L$ for \eqref{deq_brk}, $k\ge 1$.
There exists a $C^k$ function
${\hat\beta}={\hat\beta}(\beta,r,w_1,w_2,\xi)$ such that \beq{a_beta_r0}
  {\hat\beta}= \beta
  \quad\text{when }r=0
\] and $({\hat\beta},r,\kappa,w_1,w_2,\xi)$
is a change of coordinates near $\mathcal P_L$
satisfying \begin{align}
  &W^s(\mathcal P_L)
  = \{({\hat\beta},r,\kappa,w_1,w_2,\xi):
  {\hat\beta}=\rho_3,\kappa=0\}  \label{WsPL_ark}\\
  &W^u(\mathcal P_L)
  = \{({\hat\beta},r,\kappa,w_1,w_2,\xi): r=0\}.  \label{WuPL_ark}
\end{align} Hence \eqref{deq_brk} is converted into \beq{deq_ark}
  &\dot{{\hat\beta}}
  = \big(\tfrac{-2}3\rho_3(\rho_3^2-3)+ h_1\big)(\tilde\beta-\rho_3)\\
  &\dot{r}= \big(\tfrac{-\rho_3^3}6+ h_2\big)r\\
  &\dot{\kappa}= \big(\tfrac{\rho_3^3}6+h_3\big)\kappa\\
  &\dot{w}_1= h_4\\
  &\dot{w}_2= -\kappa\\
  &\dot{\xi}= \kappa r^2,
\] where $h_i$ are $C^{k-1}$ functions
satisfying $h_1,h_2,h_3=O(|({\hat\beta},r,\kappa)|)$
and $h_4=O(|r|\cdot |({\hat\beta},\kappa)|)$
as $|({\hat\beta},\kappa,r)|\to 0$,
and the projection $\pi^s_{\mathcal P_L}$ into $\mathcal P_L$ along stable fibers is \beq{pi_WsPL}
  \pi^s_{\mathcal P_L}(0,r,0,w_1,w_2,\xi)= (0,0,0,w_1,w_2,\xi).
\]
\end{prop}

\begin{proof}
At each point of 
$\mathcal P_L$
defined in \eqref{def_pl},
the linearized system corresponds to the matrix
represented in $(\beta,r,\kappa)$-coordinates as \beq{matrix_PL}
  \begin{pmatrix}
    \tfrac{-2}3\rho_3(\rho_3^2-3)&\tfrac{-1}2\rho_3\xi&0\\
    0&\tfrac{-1}6\rho_3^3&0\\
    0&0&\tfrac16\rho_3^3
  \end{pmatrix},
\] which has eigenvalues \beq{eigval_PL}
  \tfrac{-2}3\rho_3(\rho_3^2-3)
  >0,\quad
  \tfrac{-1}6\rho_3^3<0,\quad
  \tfrac{1}6\rho_3^3>0,
\] and eigenvectors \beq{eigvec_PL}
  (1,0,0)^\top,\quad
  \big(\tfrac{1}2\rho_3\xi,\tfrac{-2}3\rho_3(\rho_3^2-3)+\tfrac16\rho_3^3,0\big)^\top,\quad
  (0,0,1)^\top.
\]
Since the sets $\{r=0\}$ and $\{\kappa=0\}$
are invariant under \eqref{deq_brk},
it follows that \beq{WuPL_brk}
  W^u(\mathcal P_L)= \{(\beta,r,\kappa,w_1,w_2): r=0\}
\] and $W^s(\mathcal P_L)$ can be parameterized as \beq{WsPL_brk}
  W^s(\mathcal P_L)
  = \big\{(\beta,r,\kappa,w_1,w_2,\xi):
     \kappa=0
     \;\text{and}\;
    \beta= \rho_3+ \phi(r,w_1,w_2,\xi)r
  \big\}
\] where $\phi$ is a $C^k$ function satisfying \beq{phi_linear}
  \phi(r,w_1,w_2,\xi)
  = \frac{\tfrac{1}2\rho_3\xi}{\tfrac{-2}3\rho_3(\rho_3^2-3)+\tfrac16\rho_3^3}
  + O(r).
\] Set \beq{def_a}
  {\hat\beta}= \beta- \phi(r,w_1,w_2,\xi)r.
\] Then \eqref{phi_linear} implies \eqref{a_beta_r0}.
Now \eqref{WsPL_ark} follows from \eqref{WsPL_brk} and \eqref{def_a},
and \eqref{WuPL_ark} follows from \eqref{WuPL_brk}.
\end{proof}

A similar result holds for $\mathcal P_R$. We omit it here.

\subsection{Transversal Intersections}
\label{subsec_transversal}

\begin{figure}[t]
\centering
\begin{parbox}{.78\textwidth}{\centering
{\includegraphics[trim = 2cm 7.4cm 1.6cm 7.45cm, clip, width=.76\textwidth]{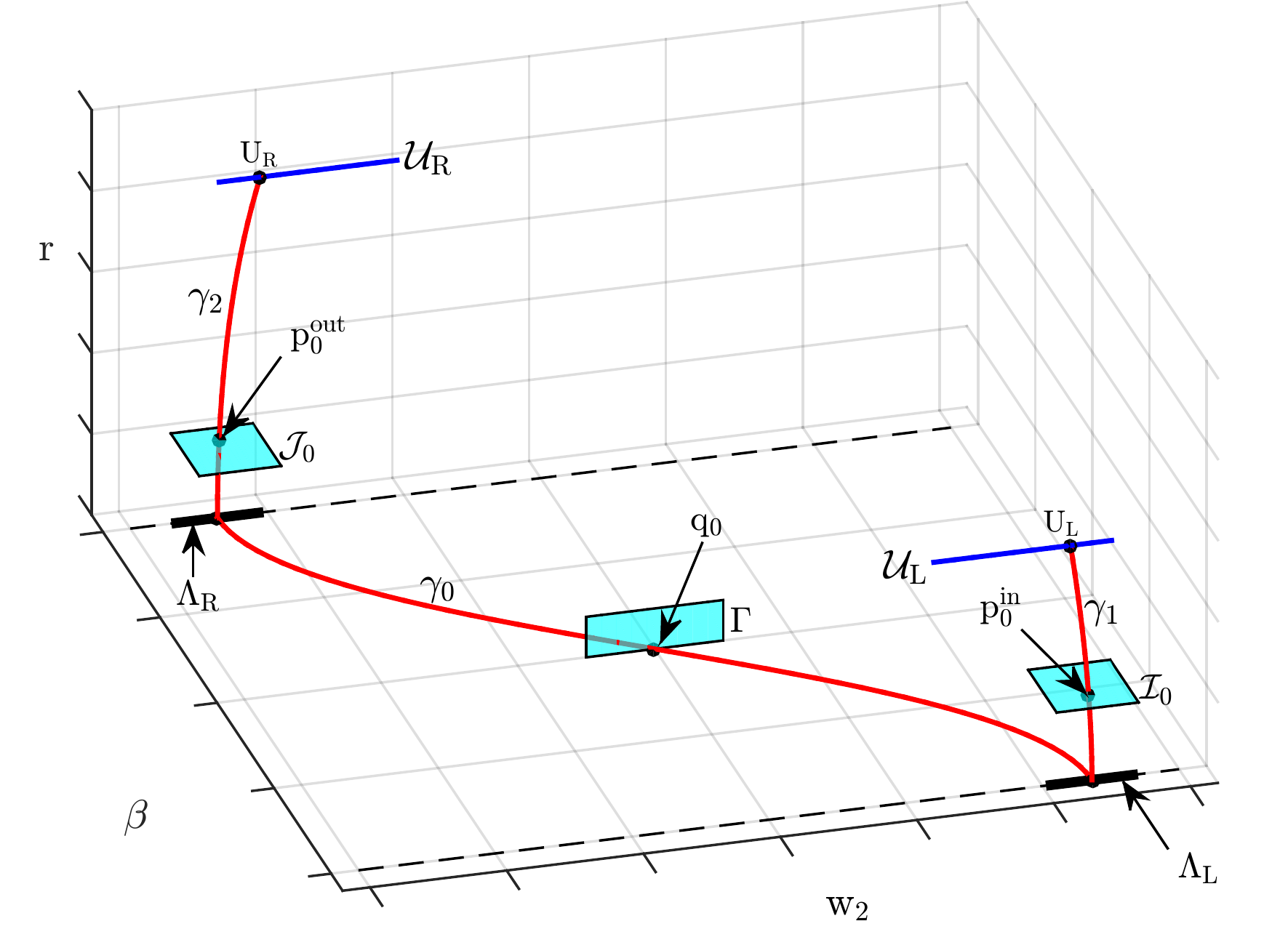}}
}\end{parbox}
\caption{
The $1$D intervals $\Lambda_L$ and $\Lambda_R$
are projections of $\mathcal I_0$ and $\mathcal J_0$, respectively,
on the critical manifolds $\mathcal P_L$ and $\mathcal P_R$.
In the $5$D space $\{r=0\}$,
the $3$D manifolds $W^u(\Lambda_L)$ and $W^s(\Lambda_R)$
intersect transversally at $q_0$,
and their intersection is the curve $\gamma_0$,
which is transversal to $\Gamma$.}
\label{fig_connection}
\end{figure}

To prove the Main Theorem,
we need to find trajectories for \eqref{deq_brk} connecting $\mathcal U_L$ and $\mathcal U_R$
in $(\beta,r,\kappa,w_1,w_2,\xi)$-space
satisfying $r\kappa=\epsilon$ for each small $\epsilon>0$.
Note that the trajectories $\gamma_1$ and $\gamma_2$
given in Proposition \ref{prop_sss}
satisfy $\gamma_1\subset W^u_0(\mathcal U_L)\cap W^s(\mathcal P_L)$
and $\gamma_2\subset W^s_0(\mathcal U_R)\cap W^u(\mathcal P_R)$.
Our strategy is to first define
$2$-dimensional manifolds
$\mathcal I_\epsilon$ and $\mathcal J_\epsilon$, $\epsilon\in [0,\epsilon_0]$,
contained in $W^u_\epsilon(\mathcal U_L)$ and $W^s_\epsilon(\mathcal U_R)$, respectively,
such that $\cup_\epsilon \mathcal I_\epsilon$ and $\cup_\epsilon \mathcal J_\epsilon$
are transverse to $\gamma_1$ and $\gamma_2$,
and then track forward/backward trajectories from $\mathcal I_\epsilon$ and $\mathcal J_\epsilon$.
An illustration with $\epsilon=0$ is shown in Fig \ref{fig_connection}.

To track trajectories evolving from $\mathcal I_\epsilon$ and $\mathcal J_\epsilon$,
we will apply the Corner Lemma
stated in Section \ref{subsec_corner}.
The key idea is to show that
the manifolds that evolve from $\mathcal I_\epsilon$ and $\mathcal J_\epsilon$,
denoted by $\mathcal I_\epsilon^*$ and $\mathcal J_\epsilon^*$, respectively,
are $C^1$ close to $W^u(\Lambda_L)$ and $W^s(\Lambda_R)$,
where $\Lambda_L\subset \mathcal P_L$ and $\Lambda_R\subset \mathcal P_R$
are projections of $\mathcal I_0$ and $\mathcal J_0$.
Hence transversal intersection of $W^u(\Lambda_L)$ and $W^s(\Lambda_R)$
will imply that of $\mathcal I_\epsilon^*$ and $\mathcal J_\epsilon^*$.

Fix a small $r^0>0$ so that $\gamma_1$ intersects $\{r=r^0\}$ at a unique point.
Denote this point by $p^\din_0$. That is, \beq{def_pin0}
  p^\din_0= \gamma_1\cap \{r=r^0\}.
\] We set \beq{def_Ieps}
  \mathcal I_\epsilon
  = W^u_\epsilon(\mathcal U_L)\cap \{r=r^0\}\cap V_1,
\] where $V_1$ is an open neighborhood of $p^\din_0$ to be specified below:
From the expression \eqref{def_ul}, $\mathcal U_L$ is $1$-dimensional,
so from $\mathrm{(H1)}$ we see that $W^u_\epsilon(\mathcal U_L)$ is $3$-dimensional.
Hence we can choose $V_1$ so that
 $\mathcal I_\epsilon$ is parametrized,
in $({\hat\beta},r,\kappa,w_1,w_2,\xi)$-coordinates given in Proposition \ref{prop_pl},
by \beq{Ieps_para}
  \mathcal I_\epsilon
  &= \big\{
  ({\hat\beta},r,\kappa,w_1,w_2,\xi): r=r^0,\kappa=\epsilon/r^0,\\
  &\qquad\qquad
  (w_1,w_2,\xi)= (w_{1L},w_{2L},s)+ \alpha_1(-u_{1L},-u_{2L},1)
  + \epsilon\theta({\hat\beta},\alpha_1,\epsilon),\\
  &\qquad\qquad
  |{\hat\beta}-\rho_3|<\Delta_1,|\alpha_1|<\Delta_1
  \big\}
\] for some $\Delta_1>0$ and some $C^4$ function $\theta$.
(The order of differentiability of $\theta$ is chosen so that the Corner Lemma applies.)
Note that $\mathcal I_0$ is a affine surface,
and $\mathcal I_\epsilon$ can be viewed as a perturbation of $\mathcal I_0$.

Let \beq{def_I}
  \mathcal I= \bigcup_{\epsilon\in [0,\epsilon_0]} \mathcal I_\epsilon.
\] Since $p^\din_0\in \gamma_1\subset W^u(\mathcal U_L)\cap W^s(\mathcal P_L)$,
frrm \eqref{WsPL_brk} and \eqref{Ieps_para}
we see that $\mathcal I$ and $W^s(\mathcal P_L)$ intersect transversally at $p^\din_0$,
and the projection into $\mathcal P_L$ of their intersection
along stable fibers is, by \eqref{pi_WsPL}, \beq{LambdaL_para}
  \Lambda_L
  &= \{(\beta,r,\kappa,w_1,w_2,\xi):\;
  \beta= \rho_3,r=0,\kappa=0,\\
  &\qquad\qquad
  (w_1,w_2,\xi)
  = (w_{1L},w_{2L},s)+\alpha_1(-u_{1L},-u_{2L},1),\\
  &\qquad\qquad
  |\alpha_1|<\Delta_1\}.
\] Also we let $p^\dout_0$, $\mathcal J_\epsilon$,
$\mathcal J$ and $\Lambda_R$ be analogously defined.

Since $\Lambda_L$ is a subset of the
normally hyperbolic critical manifold $\mathcal P_L$ for \eqref{deq_ark},
the unstable manifold $W^u(\Lambda_L)$ can be defined in the natural way.
From \eqref{WuPL_ark} we see that $W^u(\Lambda_L)\subset\{r=0\}$.
Similarly, $\Lambda_R$ and $W^s(\Lambda_R)$ are defined,
and $W^s(\Lambda_R)\subset\{r=0\}$.
Note that the trajectory $\gamma_0$ given in Proposition \ref{prop_gamma0} satisfies \[
  \gamma_0\subset W^u(\Lambda_L)\cap W^s(\Lambda_R).
\]
To track the intersection of $W^u(\Lambda_L)$ and $W^s(\Lambda_R)$ along $\gamma_0$,
we fix a hyperplane \beq{def_Gamma}
  \Gamma= \{(\beta,r,\kappa,w_1,w_2,\xi): \beta= 0\}
\] and set \beq{def_q0}
  q_0
  =\gamma_0\cap \Gamma.
\] 

\begin{prop}\label{prop_transversal_lambda}
$W^u(\Lambda_L)$ and $W^s(\Lambda_R)$
intersect transversally at $q_0$ in the space $\{r=0\}$,
and their intersection near $q_0$
is a portion of the curve $\gamma_0$ given in Proposition \ref{prop_gamma0},
and hence is transverse to $\Gamma$ at $\gamma_0$.
\end{prop}

\begin{proof}
From Proposition \ref{prop_wbk} and \ref{prop_gamma0},
we have \beq{tangent_lambdaL}
  T_{q_0}W^u(\Lambda_L)
  = \mathrm{Span} \left\{
  \begin{pmatrix}1\\0\\ *\\0\\ *\\0\end{pmatrix},
  \begin{pmatrix}0\\0\\1\\0\\2{\iota_3}(0)\\0\end{pmatrix},
  \begin{pmatrix}0\\0\\0\\-u_{1L}\\-u_{2L}\\1\end{pmatrix}
  \right\}
\] and \beq{tangent_lambdaR}
  T_{q_0}W^s(\Lambda_R)
  = \mathrm{Span} \left\{
  \begin{pmatrix}1\\0\\ *\\0\\ *\\0\end{pmatrix},
  \begin{pmatrix}0\\0\\1\\0\\-2{\iota_3}(0)\\0\end{pmatrix},
  \begin{pmatrix}0\\0\\0\\-u_{1R}\\-u_{2R}\\1\end{pmatrix}
  \right\}
\]
in $(\beta,r,\kappa,w_1,w_2,\xi)$ coordinates,
where ${\iota_3}(\sigma)$ is the positive function defined in \eqref{def_q2q3}.
Since ${\iota_3}(0)\ne 0$ and $u_{1L}\ne u_{1R}$,
from \eqref{tangent_lambdaL} and \eqref{tangent_lambdaR}
we see that $T_{q_0}W^u(\Lambda_L)$
and $T_{q_0}W^s(\Lambda_R)$
span $(\beta,\kappa,w_1,w_2,\xi)$-space
and they have a $1$-dimensional intersection
which is transversal to $\Gamma$.
Since $q_0\in \gamma_0\subset W^u(\Lambda_L)\cap W^s(\Lambda_R)$,
the desired result follows.
\end{proof}

Now we have obtained the
singular configuration $\gamma_1\cup \gamma_0\cup \gamma_2$,
which joins the end states $u_L$ and $u_R$.
In the next section we will show that there are solutions of \eqref{sf_u}
which are close to the singular configuration.

\section{Completing the Proof of the Main Theorem}
\label{sec_pf_main}
We split the proof of the main theorem into two parts.
In Section \ref{sub{KK_sec_complete}istence} we prove
the existence of solutions of
the boundary value problem \eqref{dafermos_similar} and \eqref{bc_u_infty}.
In Section \ref{subsec_convergence}
we derive the weak convergence \eqref{limit_u12}.

\subsection{Existence of Viscous Profile}
\label{sub{KK_sec_complete}istence}

\begin{prop}\label{prop_qeps}
Let $p^\din_0$, $p^\dout_0$, $q_0$,
$\mathcal I_\epsilon$, $\mathcal J_\epsilon$ and $\Gamma$
be defined in Section \ref{subsec_transversal}.
For each small $\epsilon>0$,
there exist $p^\din_\epsilon\in \mathcal I_\epsilon$,
$p^\dout_\epsilon\in \mathcal J_\epsilon$,
$q_\epsilon\in \Gamma$ and $T_{1\epsilon},T_{2\epsilon}>0$
such that \beq{link_qeps}
  q_\epsilon= p^\din_\epsilon\cdot T_{1\epsilon},\quad
  q_\epsilon= p^\dout_\epsilon\cdot (-T_{2\epsilon}),
\] where $\cdot$ denotes the flow for \eqref{deq_brk},
satisfying \beq{limit_qeps}
  (p^\din_\epsilon, p^\dout_\epsilon, q_\epsilon)
  = (p^\din_0, p^\dout_0, q_0)+ o(1)
\] as $\epsilon\to 0$, and \beq{est_T12eps}
  C^{-1}\log\frac1{\epsilon}\le T_{i\epsilon}\le C\log\frac1{\epsilon},\quad i=1,2,
\] for some $C>0$.
Moreover, if we set $\beta_\epsilon(\sigma)$ and $\kappa_\epsilon(\sigma)$
to be the $\beta$- and $\kappa$-coordinates of $q_\epsilon\cdot \sigma$,
$\sigma\in [-T_{1\epsilon},T_{2\epsilon}]$,
then \beq{est_kappa_eps}
  \max_{\sigma\in [-T_{1\epsilon},T_{2\epsilon}]}\kappa_\epsilon(\sigma)
  = \kappa_0+ o(1)
\] and \beq{bkeps_minmax}
  \max_{\sigma\in [-T_{1\epsilon},T_{2\epsilon}]}
  \pm\beta_\epsilon(\sigma)\kappa_\epsilon(\sigma)
  = \omega_0+ o(1),
\] as $\epsilon\to 0$, where $\kappa_0$ is defined in \eqref{def_kappa0},
and \beq{def_omega0}
  \omega_0
  = \kappa_0\, {\iota_2}(\sigma_0),
\] where $\sigma_0$ is the unique number such that ${\iota_1}(\sigma_0)=1$,
and ${\iota_1}(\sigma)$, ${\iota_2}(\sigma)$ are positive functions
defined in \eqref{def_q1} and \eqref{def_q2q3}.
\end{prop}

\begin{proof}
Let $\mathcal I= \bigcup_\epsilon\mathcal I_\epsilon$.
Since $\mathcal I_\epsilon\subset\{r=r^0\}$,
from the relation $\kappa=\epsilon/r$ we have \[
  \mathcal I_\epsilon
  = \mathcal I\cap \{\kappa=\epsilon/r^0\}.
\] 
From the construction of $p^\din_0$, $p^\dout_0$ and $q_0$, we have \[
  &\pi^u_{\mathcal P_L}(q_0)
  = \pi^s_{\mathcal P_L}(p^\din_0)= (0,0,0,w_{1L},w_{2L},s)\in \Lambda_L\\
  &\pi^s_{\mathcal P_R}(q_0)
  = \pi^u_{\mathcal P_R}(p^\dout_0)= (0,0,0,w_{1R},w_{2R},s)\in \Lambda_R\\
\] in $(\beta,r,\kappa,w_1,w_2,\xi)$-coordinates,
where $\pi^{s,u}_{\mathcal P_{L,R}}$ is the projection onto $\mathcal P_{L,R}$
along stable/unstable fibers.
For the system \eqref{deq_ark}, the conditions for the Corner Lemma are satisfied.
Hence there exists a neighborhood $V_0$ of $q_0$
such that \beq{IV0_C1}
  \text{
    $\mathcal{I}_\epsilon^*\cap V_0$
    is $C^1$ close to $T_{q_0}W^u(\Lambda_L)\cap V_0$,
  }
\] where $\mathcal I_\epsilon^*= \mathcal I_\epsilon\cdot [0,\infty)$.
Similarly, setting $\mathcal J_\epsilon^*= \mathcal J_\epsilon\cdot (-\infty,0]$,
we have \beq{JV0_C1}
   \text{
    $\mathcal{J}_\epsilon^*\cap V_0$
    is $C^1$ close to $T_{q_0}W^s(\Lambda_R)\cap V_0$.
  }
\]
From \eqref{IV0_C1}, \eqref{JV0_C1} and Proposition \ref{prop_transversal_lambda},
it follows that the projections of $\mathcal{I}_\epsilon^*$ and $\mathcal{J}_\epsilon^*$
in the $5$-dimensional space $\{r=0\}$
intersect transversally at a unique point in $\Gamma$ near $q_0$.
From the relation $r=\epsilon/\kappa$,
we then recover a unique intersection point \beq{def_qeps}
  q_\epsilon
  \in \mathcal{I}_\epsilon^*\cap \mathcal{J}_\epsilon^*\cap \Gamma
\] in $(\beta,r,\kappa,w_1,w_2,\xi)$-space.
By the construction we have 
\eqref{link_qeps} and \eqref{limit_qeps}.
The estimate \eqref{est_T12eps} 
follows from \eqref{est_Teps_corner}.

The unstable fiber containing $q_0$ 
in $W^u(\mathcal P_L)$ is the trajectory
$\gamma_0$ defined in Proposition \ref{prop_gamma0}.
The $\beta$- and $\kappa$-coordinates on $\gamma_0$
are ${\iota_1}(\sigma)$ and $\kappa_0{\iota_2}(\sigma)$, respectively.
From \eqref{def_q2q3} we know ${\iota_2}(\sigma)\le {\iota_2}(0)=1$.
Hence \eqref{est_kappa_eps} follows.
To prove \eqref{bkeps_minmax},
by symmetry of $\gamma_0$, it suffices to show that \beq{bk_minmax}
  \max_{\sigma\in (-\infty,0]} {\iota_1}(\sigma){\iota_2}(\sigma)
  = {\iota_2}(\sigma_0),
\] where $\sigma_0$ is defined by ${\iota_1}(\sigma_0)=1$.
Note that the values of ${\iota_1}(\sigma)$ and ${\iota_2}(\sigma)$ are positive on $(-\infty,0)$, and \[
  {\iota_1}(0){\iota_2}(0)= 0\cdot 1= 0,\quad
  {\iota_1}(-\infty){\iota_2}(-\infty)= \rho_3\cdot 0= 0.
\] By taking the derivative of ${\iota_1}(\sigma){\iota_2}(\sigma)$ it can be readily seen
that the maximum of this function
occurs at a unique number $\sigma_0$ satisfying ${\iota_1}(\sigma_0)=1$.
Indeed, from the definition \eqref{def_q1} and \eqref{def_q2q3}, we have \[
  \frac{d}{d\sigma}\big( {\iota_1}(\sigma){\iota_2}(\sigma)\big)
  &={\iota_2}(\sigma)\left[
    \dot{q}_1(\sigma)+ \tfrac16{\iota_1}(\sigma)^3
  \right]\\
  &=\tfrac16{\iota_2}(\sigma)\left[
    -(\beta^4-6\beta^2+6)+\beta^3
  \right],
\] where we write $\beta={\iota_1}(\sigma)$.
Since $0<{\iota_1}(\sigma)<\rho_3$ for $\sigma\in (-\infty,0)$,
this derivative has a unique zero,
which occurs when $\beta=1$.
This proves \eqref{bk_minmax} and hence \eqref{bkeps_minmax}.
\end{proof}

\begin{prop}\label{prop_existence}
Let $q_\epsilon= (\beta_\epsilon^0,r_\epsilon^0,\kappa_\epsilon^0,
w_{1\epsilon}^0,w_{2\epsilon}^0,\xi_\epsilon^0)\in \Gamma$
be defined in Proposition \ref{prop_qeps}.
Let $(u_{1\epsilon}^0,u_{2\epsilon}^0)= (\beta_\epsilon^0/r_\epsilon^0,1/(r_\epsilon^0)^2)$
and \beq{def_utilde}
  (\tilde{u}_{1\epsilon},\tilde{u}_{2\epsilon},
  \tilde{w}_{1\epsilon},\tilde{w}_{2\epsilon})(\xi)
  = (u_{1\epsilon}^0,u_{2\epsilon}^0,w_{1\epsilon}^0,w_{2\epsilon}^0)
  \myflow{dafermos_similar_xi}(\xi-\xi_\epsilon^0),
\] or equivalently, \beq{def_utilde2}
  (\tilde{u}_{1\epsilon},\tilde{u}_{2\epsilon},
  \tilde{w}_{1\epsilon},\tilde{w}_{2\epsilon},\xi)
  = (u_{1\epsilon}^0,u_{2\epsilon}^0,w_{1\epsilon}^0,w_{2\epsilon}^0,\xi_\epsilon^0)
  \myflow{sf_u12}\left(\tfrac{\xi-\xi_\epsilon^0}{\epsilon}\right).
\] Then $(\tilde{u}_{1\epsilon},\tilde{u}_{2\epsilon})$
is a solution of \eqref{dafermos_similar} and \eqref{bc_u_infty},
and it satisfies \eqref{est_umax}.
\end{prop}

\begin{proof}
Since \eqref{dafermos_similar} and \eqref{dafermos_similar_xi} are equivalent
and $(\tilde{u}_{1\epsilon},\tilde{u}_{2\epsilon},
\tilde{w}_{1\epsilon},\tilde{w}_{2\epsilon})(\xi)$
is a solution of \eqref{dafermos_similar_xi},
we know $(\tilde{u}_{1\epsilon},\tilde{u}_{2\epsilon})$ is a solution of \eqref{dafermos_similar}.

Let $T_{1\epsilon}$ and $T_{2\epsilon}$ be as defined in Proposition \ref{prop_qeps}.
Then \[
  q_\epsilon\cdot (-T_{1\epsilon})\in \mathcal I_\epsilon,\quad
  q_\epsilon\cdot (T_{2\epsilon})\in \mathcal I_\epsilon,
\] where $\cdot$ denotes the flow for \eqref{deq_brk}.
Since $\mathcal I_\epsilon\subset W^u_\epsilon(\mathcal U_L)$
and $\mathcal J_\epsilon\subset W^s_\epsilon(\mathcal U_R)$,
from \eqref{est_dist_ul} and \eqref{est_dist_ur} it follows that \[
  \lim_{t\to -\infty}\mathrm{dist}(p_\epsilon^0\cdot t,\mathcal U_L)=0,
  \quad
  \lim_{t\to \infty}\mathrm{dist}(p_\epsilon^0\cdot t,\mathcal U_R)=0,
\] which implies \eqref{bc_u_infty}.
Since $\tilde{u}_{2\epsilon}=(1/\tilde{r}_\epsilon)^2=(\tilde{\kappa}_\epsilon/\epsilon)^2$
and $\tilde{u}_{1\epsilon}=\tilde\beta_\epsilon/\tilde{r}_\epsilon
=\tilde\beta_\epsilon\tilde{\kappa}_\epsilon/\epsilon$,
from \eqref{est_kappa_eps} and \eqref{bkeps_minmax}
we obtain \eqref{est_umax}.
\end{proof}

Here we justify the assertion made at the end of Section \ref{sec_cpt}.
From the equation for $\dot\xi$ in \eqref{deq_brk}, we have
$\dot{\tilde{\xi}}_{\epsilon}=\tilde{\kappa}_{\epsilon}\tilde{r}_{\epsilon}^2
= \epsilon^2/\tilde{\kappa}_{\epsilon}$.
Since the integral of $1/\kappa$ along any compact segment of $\gamma_0$ is finite,
the change in $\xi$ near such a segment is of order $O(\epsilon^2)$.

\subsection{Convergence of Viscous Profile}
\label{subsec_convergence}

\begin{prop}\label{prop_est_utilde}
Let $\tilde{u}_\epsilon= (\tilde{u}_{1\epsilon},\tilde{u}_{2\epsilon})$
be the solution of \eqref{dafermos_similar} and \eqref{bc_u_infty}
given in Proposition \ref{prop_existence}.
Let $p^\din_\epsilon$ and $p^\dout_\epsilon$
be defined in Proposition \ref{prop_qeps},
and $s$ defined in \eqref{def_s}.
Then \begin{align}
  &|\xi_\epsilon^\din-s|+ |\xi_\epsilon^\dout-s|= o(1)\label{est_xi_eps}\\
  &\int_{-\infty}^{\xi_\epsilon^\din}|\tilde{u}_\epsilon(\xi)-u_L|\;d\xi
  + \int_{\xi_\epsilon^\dout}^{\infty} |\tilde{u}(\xi)-u_R|\;d\xi
  = o(1)\label{int_u_outer}\\
  &\int_{\xi_\epsilon^\din}^{\xi_\epsilon^\dout} \tilde{u}_\epsilon(\xi)\;d\xi
  = (0,e_0)+ o(1)\label{int_u_inner}
\end{align} as $\epsilon\to 0$,
where $\xi^{\mathrm{in,out}}_\epsilon$ and 
are the $\xi$-coordinates of $p^{\mathrm{in,out}}_\epsilon$.
\end{prop}

\begin{proof}
Note that $s$ is the $\xi$-coordinate of $p^\din_0$.
From the triangular inequality we have  \[
  |\xi^\din_\epsilon-s|
  \le |p^\din_\epsilon-p^\din_0|
 = o(1),
\] where the second inequality follows from \eqref{limit_qeps}.
A similar inequality holds for $\xi^\dout_\epsilon$,
so we obtain \eqref{est_xi_eps}.

Since every point in $\mathcal U_L$ has $u$-coordinate equal to $u_L$, \[
  |\tilde{u}(\xi)-u_L|
  \le \mathrm{dist}\big((\tilde u(\xi),\tilde w(\xi),\xi),\mathcal U_L\big)
  = \mathrm{dist}\big(
    (u_\epsilon^0,w_\epsilon^0,\xi_\epsilon^0)
    \myflow{sf_u}\tfrac{\xi-\xi_\epsilon^0}{\epsilon},
    \mathcal U_L
  \big),
\] where the last equality follows from \eqref{def_utilde}.
Using \eqref{est_dist_ul},
the last term is $\le C\exp\big(\mu\tfrac{\xi-\xi_\epsilon^0}{\epsilon}\big)$.
Since $\xi_\epsilon^\din<\xi_\epsilon^0$, it follows that \[
  \int_{-\infty}^{\xi_\epsilon^\din}|\tilde{u}(\xi)-u_L|\;d\xi
  &\le \int_{-\infty}^{\xi_\epsilon^\din}
    C\exp\big(\mu\tfrac{\xi-\xi_\epsilon^0}{\epsilon}\big)
  \;d\xi
  &\le \int_{-\infty}^{\xi_\epsilon^\din}
    C\exp\big(\mu\tfrac{\xi-\xi_\epsilon^\din}{\epsilon}\big)
  \;d\xi
  = \frac{\epsilon}{\mu} C.
\] 
A similar inequality holds for
$\int_{\xi_\epsilon^\dout}^{\infty}|\tilde{u}(\xi)-u_R|\;d\xi$,
so we obtain \eqref{int_u_outer}.

From the equation for $\dot{\xi}$ in \eqref{deq_brk},
denoting the time variable by $\sigma$,
we can write $\xi=\xi(\sigma)$ by \beq{xi_zeta}
  \xi(0)=\xi^0_\epsilon,\quad
  \tfrac{d\xi}{d\sigma}= \tilde{\kappa}_\epsilon(\xi) \tilde{r}_\epsilon(\xi)^2,
\] From \eqref{link_qeps} we have \beq{xi_T12}
  \xi(-T_{1\epsilon})= \xi^\din_\epsilon,\quad
  \xi(T_{2\epsilon})= \xi^\dout_\epsilon.
\] From \eqref{xi_zeta} and \eqref{xi_T12},
using the equation for $\dot{w}_2$ in \eqref{deq_brk},
it follows that \beq{int_u2_inner}
  &\int_{\xi^\din}^{\xi^\dout}\tilde{u}_{2\epsilon}(\xi)\;d\xi
  = \int_{\xi^\din}^{\xi^\dout}\frac{1}{(\tilde{r}_{\epsilon}(\xi))^2}\;d\xi
  = \int_{-T_{1\epsilon}}^{T_{2\epsilon}}
  \frac{\tilde{\kappa_\epsilon}(\sigma)\tilde{r}_\epsilon(\sigma)^2}
  {\tilde{r}_\epsilon(\sigma)^2}\;d\sigma\\
  &= -\int_{-T_{1\epsilon}}^{T_{2\epsilon}}\dot{\tilde{w}}_{2\epsilon}(\sigma)\;d\sigma
  = w_2(p^\dout_\epsilon)- w_2(p^\din_\epsilon)
  = w_{2L}- w_{2R}+ o(1)
\] where $w_2(p)$ denotes the $w_2$-coordinate of $p$,
and that \beq{int_u1_inner}
  &\int_{\xi^\din}^{\xi^\dout}|\tilde{u}_{1\epsilon}(\xi)|\;d\xi
  = \int_{\xi^\din}^{\xi^\dout}
  \frac{|\tilde{\beta_\epsilon}(\xi)|}{\tilde{r}_\epsilon(\xi)}\;d\xi
  = \int_{-T_{1\epsilon}}^{T_{2\epsilon}}
  |\tilde{\beta}_\epsilon(\sigma)|\tilde{\kappa}_\epsilon(\sigma)\tilde{r}_\epsilon(\sigma)\;d\sigma\\
  &= \epsilon\int_{-T_{1\epsilon}}^{T_{2\epsilon}}
  |\tilde{\beta}_\epsilon(\sigma)|\;d\sigma
  \le C \epsilon(T_{1\epsilon}+T_{2\epsilon})
  \le \tilde{C}(\epsilon\log\tfrac1\epsilon)
  = o(1),
\] where the last inequality follows from \eqref{est_T12eps}.
Now \eqref{int_u2_inner} and \eqref{int_u1_inner} give \eqref{int_u_inner}.
\end{proof}

The remaining part of this section is analogous
to that in Section \ref{TwoPhase_subsec_convergence},
so we only sketch the proofs briefly.

\begin{prop}\label{prop_convergence_similar}
Let $\tilde{u}_\epsilon= (\tilde{u}_{1\epsilon},\tilde{u}_{2\epsilon})$
be the solution of \eqref{dafermos_similar} and \eqref{bc_u_infty}
given in Proposition \ref{prop_existence}. Then \beq{limit_utilde}
  \tilde{u}_\epsilon
  \rightharpoonup
  u_L+ (u_R-u_L)\mathrm{H}(\xi-s)+ (0,e_0)\delta(\xi-s)
\] in the sense of distributions as $\epsilon\to 0$.
\end{prop}

\begin{proof}
Given any smooth function $\psi\in C^\infty_c(\mathbb R)$ with compact support,
using Proposition \ref{prop_est_utilde} it can be readily seen that
\[
  \int_{-\infty}^\infty \psi(\xi)\tilde{u}_\epsilon(\xi)\;d\xi
  = u_L\int_{-\infty}^{s}\psi(\xi)\;d\xi
  +u_R\int_{s}^\infty \psi(\xi)\,d\xi
  + (0,e_0)\psi(s)
  + o(1).
\]
This holds for all $\psi$, so \eqref{limit_utilde} holds.
\end{proof}

\begin{prop}\label{prop_convergence}
Let $\tilde{u}_\epsilon= (\tilde{u}_{1\epsilon},\tilde{u}_{2\epsilon})$
be the solution of \eqref{dafermos_similar} and \eqref{bc_u_infty}
given in Proposition \ref{prop_existence}.
Let $u_\epsilon(x,t)= \tilde{u}_\epsilon(x/t)$.
Then the weak convergence \eqref{limit_u12} holds.
\end{prop}

\begin{proof}
Given any $\varphi\in C^\infty_c(\mathbb R\times \mathbb R_+)$,
using Proposition \ref{prop_convergence_similar} we have \[
  &\int_0^\infty\int_{-\infty}^\infty \varphi(x,t)u_\epsilon(x,t)\,dx\,dt
  =\int_0^\infty t\int_{-\infty}^\infty \varphi(t\xi,t)\tilde{u}_\epsilon(\xi)\,d\xi\,dt\\
  &=\int_0^\infty t\left\{
      u_L\int_{-\infty}^s \varphi(t\xi,t)\,d\xi
      + (0,e_0)\varphi(st,t)
      + u_R\int_s^\infty \varphi(t\xi,t)\,d\xi
  \right\}dt+ o(1)\\
  &= u_L\int_0^\infty\!\!\! \int_{-\infty}^{st}\varphi(x,t)\,dx\,dt
  +u_R\int_0^\infty\!\!\!\! \int_{st}^{\infty}\varphi(x,t)\,dx\,dt
  +(0,e_0)\int_0^\infty t\varphi(st,t)\,dt
  +o(1).
\] 
By the definition \eqref{def_delta},
this implies \eqref{limit_u12}.
\end{proof}

Proposition \ref{prop_existence} and \ref{prop_convergence}
complete the proof of the Main Theorem.

\section*{acknowledgements}
The author would express his gratitude to his advisor,
Professor Babara L.\ Keyfitz,
without whose help this work would not have been possible.


\begin{thebibliography}{Sch08b}

\bibitem[Den90]{Deng:1990}
Bo~Deng.
\newblock Homoclinic bifurcations with nonhyperbolic equilibria.
\newblock {\em SIAM J. Math. Anal.}, 21(3):693--720, 1990.

\bibitem[Fen77]{Fenichel:1977}
Neil Fenichel.
\newblock Asymptotic stability with rate conditions. {II}.
\newblock {\em Indiana Univ. Math. J.}, 26(1):81--93, 1977.

\bibitem[Fen79]{Fenichel:1979}
Neil Fenichel.
\newblock Geometric singular perturbation theory for ordinary differential
  equations.
\newblock {\em J. Differential Equations}, 31(1):53--98, 1979.

\bibitem[Fen74]{Fenichel:1973}
Neil Fenichel.
\newblock Asymptotic stability with rate conditions.
\newblock {\em Indiana Univ. Math. J.}, 23:1109--1137, 1973/74.

\bibitem[Hsu15]{Hsu:2015a}
Ting-Hao Hsu.
\newblock Viscous singular shock profiles for a system of conservation laws
  modeling two-phase flow.
\newblock {\em Under review. Available at
  \url{http://arxiv.org/abs/1512.00394/}}.

\bibitem[Jon95]{Jones:1995}
Christopher K. R.~T. Jones.
\newblock Geometric singular perturbation theory.
\newblock In {\em Dynamical systems ({M}ontecatini {T}erme, 1994)}, volume 1609
  of {\em Lecture Notes in Math.}, pages 44--118. Springer, Berlin, 1995.

\bibitem[JT09]{Jones:2009}
Christopher K. R.~T. Jones and Siu-Kei Tin.
\newblock Generalized exchange lemmas and orbits heteroclinic to invariant
  manifolds.
\newblock {\em Discrete Contin. Dyn. Syst. Ser. S}, 2(4):967--1023, 2009.

\bibitem[Key11]{Keyfitz:2011}
Barbara~Lee Keyfitz.
\newblock Singular shocks: retrospective and prospective.
\newblock {\em Confluentes Math.}, 3(3):445--470, 2011.

\bibitem[KK89]{Keyfitz:1989}
Barbara~Lee Keyfitz and Herbert~C. Kranzer.
\newblock A viscosity approximation to a system of conservation laws with no
  classical riemann solution.
\newblock In {\em Nonlinear hyperbolic problems (Bordeaux, 1988)}, volume 1402
  of {\em Lecture Notes in Math.}, pages 185--197. Springer, Berlin, 1989.

\bibitem[KK90]{Kranzer:1990}
Herbert~C. Kranzer and Barbara~Lee Keyfitz.
\newblock A strictly hyperbolic system of conservation laws admitting singular
  shocks.
\newblock In {\em Nonlinear evolution equations that change type}, volume~27 of
  {\em IMA Vol. Math. Appl.}, pages 107--125. Springer, New York, 1990.

\bibitem[KK95]{Keyfitz:1995}
Barbara~Lee Keyfitz and Herbert~C. Kranzer.
\newblock Spaces of weighted measures for conservation laws with singular shock
  solutions.
\newblock {\em J. Differential Equations}, 118(2):420--451, 1995.

\bibitem[KT12]{Keyfitz:2012}
Barbara~Lee Keyfitz and Charis Tsikkou.
\newblock Conserving the wrong variables in gas dynamics: a {R}iemann solution
  with singular shocks.
\newblock {\em Quart. Appl. Math.}, 70(3):407--436, 2012.

\bibitem[Liu04]{Liu:2004}
Weishi Liu.
\newblock Multiple viscous wave fan profiles for {R}iemann solutions of
  hyperbolic systems of conservation laws.
\newblock {\em Discrete Contin. Dyn. Syst.}, 10(4):871--884, 2004.

\bibitem[Sch04]{Schecter:2004}
Stephen Schecter.
\newblock Existence of {D}afermos profiles for singular shocks.
\newblock {\em J. Differential Equations}, 205(1):185--210, 2004.

\bibitem[Sch08a]{Schecter:2008}
Stephen Schecter.
\newblock Exchange lemmas. {I}. {D}eng's lemma.
\newblock {\em J. Differential Equations}, 245(2):392--410, 2008.

\bibitem[Sch08b]{Schecter:2008a}
Stephen Schecter.
\newblock Exchange lemmas. {II}. {G}eneral exchange lemma.
\newblock {\em J. Differential Equations}, 245(2):411--441, 2008.

\bibitem[Sev07]{Sever:2007}
Michael Sever.
\newblock Distribution solutions of nonlinear systems of conservation laws.
\newblock {\em Mem. Amer. Math. Soc.}, 190(889):viii+163, 2007.

\bibitem[SSS93]{Schaeffer:1993}
David~G. Schaeffer, Stephen Schecter, and Michael Shearer.
\newblock Non-strictly hyperbolic conservation laws with a parabolic line.
\newblock {\em J. Differential Equations}, 103(1):94--126, 1993.

\bibitem[Wig94]{Wiggins:1994}
Stephen Wiggins.
\newblock {\em Normally hyperbolic invariant manifolds in dynamical systems},
  volume 105 of {\em Applied Mathematical Sciences}.
\newblock Springer-Verlag, New York, 1994.
\newblock With the assistance of Gy{{\"o}}rgy Haller and Igor Mezi{{\'c}}.

\end{thebibliography}

\end{document}